\documentclass[12pt,a4paper]{article}

\usepackage{graphicx}
\usepackage{subfig}
\usepackage{amsmath,amsthm,amssymb}
\usepackage{esint}
\usepackage{mathrsfs}
\usepackage{appendix}

\usepackage{hyperref}

\usepackage[left=2.2cm,right=2.2cm,top=3cm,bottom=3.5cm]{geometry}

\usepackage{color}
\usepackage{array}
\usepackage{url}

\usepackage[normalem]{ulem}

\newtheorem{theorem}{Theorem}[section]
\newtheorem*{theorem*}{Theorem}
\newtheorem{lemma}[theorem]{Lemma}

\newtheorem{proposition}[theorem]{Proposition}

\theoremstyle{definition}
\newtheorem{definition}[theorem]{Definition}

\theoremstyle{remark}
\newtheorem{remark}[theorem]{Remark}
\usepackage{authblk}

\usepackage{mathtools}
\usepackage{amssymb}
\usepackage{appendix}

\numberwithin{equation}{section}

\usepackage[numbers,sort&compress]{natbib}

\DeclareMathOperator{\supp}{supp}

\begin{document}
	
\title{On Energy Conservation for the Hydrostatic Euler Equations: An Onsager Conjecture}
	
\author[]{Daniel W. Boutros\footnote{Department of Applied Mathematics and Theoretical Physics, University of Cambridge, Cambridge CB3 0WA UK. Email: \textsf{dwb42@cam.ac.uk}}\space,  Simon Markfelder\footnote{Institute of Mathematics, University of W{\"u}rzburg, Emil-Fischer-Str. 40, 97074 W{\"u}rzburg, Germany. Email: \textsf{simon.markfelder@gmail.com}} \space and Edriss S. Titi\footnote{Department of Mathematics, Texas A\&M University, College Station, TX 77843-3368, USA; Department of Applied Mathematics and Theoretical Physics, University of Cambridge, Cambridge CB3 0WA UK; also Department of Computer Science and Applied Mathematics, Weizmann Institute of Science, Rehovot 76100, Israel. Emails: \textsf{titi@math.tamu.edu} \; \textsf{Edriss.Titi@maths.cam.ac.uk} \; \textsf{edriss.titi@weizmann.ac.il}}}

\date{June 13, 2023}

\maketitle

\begin{abstract}
Onsager's conjecture, which relates the conservation of energy to the regularity of weak solutions of the Euler equations, was completely resolved in recent years. In this work, we pursue an analogue of Onsager's conjecture in the context of the \textit{hydrostatic} Euler equations (also known as the inviscid primitive equations of oceanic and atmospheric dynamics). In this case the relevant conserved quantity is the horizontal kinetic energy.

We first consider the standard notion of weak solution which is commonly used in the literature. We show that if the horizontal velocity $(u,v)$ is sufficiently regular then the horizontal kinetic energy is conserved. Interestingly, the spatial H\"older regularity exponent which is sufficient for energy conservation in the context of the hydrostatic Euler equations is $\frac{1}{2}$ and hence larger than the corresponding regularity exponent for the Euler equations (which is $\frac{1}{3}$).

This is due to the anisotropic regularity of the velocity field: Unlike the Euler equations, in the case of the hydrostatic Euler equations the vertical velocity $w$ is one degree spatially less regular with respect to the horizontal variables, compared to the horizontal velocity $(u,v)$. Since the standard notion of weak solution is not able to deal with this anisotropy properly, we introduce two new notions of weak solutions for which the vertical part of the nonlinearity is interpreted as a paraproduct. We finally prove several sufficient conditions for such weak solutions to conserve energy.
	
\end{abstract}

\noindent \textbf{Keywords:} Onsager's conjecture, energy conservation, hydrostatic Euler equations, primitive equations, weak solutions

\vspace{0.1cm} \noindent \textbf{Mathematics Subject Classification:} 76B99 (primary), 35Q35, 35D30 (secondary)

	
\bibliographystyle{abbrv}
\section{Introduction}
In this paper, we study the three-dimensional hydrostatic Euler equations with rotation (also known as the reduced inviscid primitive equations of planetary scale oceanic and atmospheric dynamics, i.e. without coupling to temperature), which are given \cite{cao}
\begin{align}
&\partial_t u + \nabla \cdot (u \textbf{u}) - \Omega v + \partial_x p = 0,  \label{equationu} \\
&\partial_t v + \nabla \cdot (v \textbf{u}) + \Omega u + \partial_y p = 0, \label{equationv} \\
&\partial_z p = 0, \label{equationp} \\
&\nabla \cdot \mathbf{u} = 0. \label{equationw}
\end{align}
Here $\mathbf{u} = (u,v,w)$ is the velocity field, $p$ is the pressure and $\Omega \in \mathbb{R}$ is the Coriolis parameter. We will also denote by $\textbf{u}_H = (u,v)$ the horizontal velocity vector field for notational convenience.

It is already apparent from equation \eqref{equationw} that the vertical velocity $w$ will have one degree less spatial regularity in $x$ and $y$ than the horizontal velocities $u$ and $v$. The fact that the regularity of the velocity vector field is anisotropic will keep appearing when we study regularity criteria which are sufficient for energy conservation.

In terms of the domain, the physical boundary-value problem for the system \eqref{equationu}-\eqref{equationw} is usually studied in a three-dimensional channel
\begin{equation*}
M = \{ (x,y,z) \in \mathbb{R}^3 : 0 \leq z \leq L, (x,y) \in \mathbb{R}^2 \}
\end{equation*}
subject to the no-normal flow boundary condition in the vertical direction
\begin{equation} \label{nonormalbc}
w(x,y,0,t) = w (x,y,L,t) = 0.
\end{equation}
Furthermore, $u,v,w$ and $p$ are periodic of period $2 L$ in $x$ and $y$. We consider the above system with the initial conditions
\begin{equation} \label{inconditions}
u(x,y,z,0) = u_0 (x,y,z), \quad v(x,y,z,0) = v_0 (x,y,z).
\end{equation}
System \eqref{equationu}-\eqref{inconditions} is ill-posed in any Sobolev space \cite{ibrahim}, while it is locally well-posed in the space of analytic functions \cite{ghoul,kukavicatemamvicolziane}.

For classical (analytic) solutions of the equations in the channel, we can extend the solution in the $z$-direction by extending the vertical velocity $w$ to be odd in $z$ and the horizontal velocities $u$ and $v$ to be even in $z$ with periodic domain $[-L, L]$. The odd extension of $w$ is made possible by the no-normal flow boundary condition \eqref{nonormalbc}.

For classical solutions, these symmetry conditions are invariant under the equations \cite{ibrahim,ghoul}. Without loss of generality we take $L=\frac{1}{2}$ which makes the fundamental periodic domain to be the flat unit torus $\mathbb{T}^3$. Therefore for classical solutions, studying the physical boundary-value problem in the channel with the no-normal flow boundary conditions is equivalent to posing the problem on the torus subject to the symmetry conditions given above for $\mathbf{u}$ and $p$.

For this reason, throughout the paper, we will study weak solutions of the primitive equations on the torus with periodic boundary conditions and subject to the symmetry conditions. Notably, if such a weak solution satisfies suitable regularity assumptions such that one can make sense of the trace when $z = 0, L$, then the odd extension of $w$ actually satisfies the original physical boundary condition \eqref{nonormalbc}. Consequently, such a weak solution solves the original boundary value problem in $M$. In all the cases we consider in this paper, the weak solutions have sufficient regularity to interpret boundary condition \eqref{nonormalbc} in a trace sense.

One can observe that the horizontal kinetic energy
\begin{equation*}
E (t) = \int_{\mathbb{T}^3} \bigg( \lvert u (x,t) \rvert^2 + \lvert v (x,t) \rvert^2 \bigg) d x
\end{equation*}
is a formally conserved quantity (i.e. it is conserved by classical solutions). From now on, we will refer to this conserved quantity as the kinetic energy or simply the energy.

In this paper, we will show, loosely speaking, that the kinetic energy of weak solutions is conserved if $u$ and $v$ are H\"older continuous with exponent bigger than $\frac{1}{2}$ (we do not assume this for $w$). This is in contrast to the corresponding result for the Euler equations, where the critical exponent is $\frac{1}{3}$ (for $u, v$ and $w$). One might plausibly explain that the H\"older exponent being higher in this case is due to the lower spatial regularity of the vertical velocity $w$.

The primitive equations model the planetary scale dynamics of the ocean and the atmosphere.
These equations can be derived as an asymptotic limit of the small aspect ratio (the ratio of height to depth in a fluid) from the Rayleigh-B\'enard (Boussinesq) system. These equations have been the subject of a lot of attention over the years. So the overview that will be given below will necessary be incomplete in summarising all the relevant work.

The primitive equations were first introduced in \cite{richardson}. The viscous primitive equations were studied in a mathematical context for the first time in \cite{lionsjl1,lionsjl2,lionsjl3} (in which the global existence of weak solutions was established). The short time existence and uniqueness of strong solutions was established in \cite{guillen} (see also \cite{temamreview}).  The global existence of strong solutions was then proven in \cite{cao} (see also \cite{kobelkov}), with results for different boundary conditions in \cite{kukavica1,kukavica2}. In \cite{hieber} the case with less regular initial data was investigated using a semigroup method. The uniqueness of weak solutions was discussed in \cite{medjo,ju1,kukavicacontinuous,lidiscontinuous,breschuniqueness,bresch2D,ju2}.

The inviscid primitive equations have also been the subject of many works. The linear ill-posedness in Sobolev spaces for the inviscid case without rotation (i.e. $\Omega = 0$) was established in \cite{renardy}. Subsequently, the nonlinear ill-posedness was proven in \cite{han}.
Then it was shown in \cite{cao4,wong} that smooth solutions to the inviscid primitive equations without rotation can develop a finite-time singularity. These results were extended to the case with rotation in \cite{ibrahim}. Then in \cite{ghoul} it was proved that the lifespan of solutions can be prolonged by fast rotation for ``well-prepared'' initial data. The papers \cite{brenierhomogeneous,brenierremarks,grenier,masmoudi,kukavicalocal,kukavicatemamvicolziane,kukavicatemamvicolziane2,ghoul} investigated the local well-posedness for analytic initial data.

The nonuniqueness (and global existence) of weak solutions with initial data in $L^\infty$ for the inviscid primitive equations was established in \cite{chiodaroli}. The paper \cite{chiodaroli} also showed the non-conservation of energy. To the knowledge of the authors, it is the only paper that uses convex integration for these equations. It is then natural to consider an analogue of Onsager's conjecture for energy conservation in the hydrostatic Euler equations, which is the subject matter of this contribution.

Onsager's conjecture was first posed in \cite{onsager} for the Euler equations. It states that if a weak solution $\mathbf{u}$ of the Euler equations belongs to $L^3 ((0,T) ; C^{0, \alpha} (\mathbb{T}^3))$ with $\alpha > \frac{1}{3} $ then the spatial $L^2$ norm of $\mathbf{u}$ is conserved. In this context, the $L^2$ norm of the velocity field $\mathbf{u}$ has the physical interpretation of kinetic energy. The second half of the conjecture states that there exist weak solutions of the Euler equations for $\alpha < \frac{1}{3}$ which do not conserve energy. Yet, it was also shown in \cite{bardosloss} that there exist weak solutions that only have $L^2$ spatial regularity and which do conserve energy.

The importance of Onsager's conjecture was emphasised in \cite{eyink}, which gave a proof of the first half of the conjecture under a slightly more strict assumption than $C^{0,\alpha}$, for $\alpha > \frac{1}{3}$. A complete proof of the first half was then given in \cite{constantin}. A different approach was introduced in \cite{duchon} to prove the result, by deriving an equation of local energy balance with a defect term (which describes the energy flux at the vanishing limit of small scales) to capture a possible lack of regularity. The first half of the conjecture in the presence of wall boundaries was proven in \cite{titi2018} (and references therein).

The second half of Onsager's conjecture has also been proven by now. The first result in this direction was by \cite{scheffer}, which proved that there exist nontrivial weak solutions of the Euler equation with compact support in time (which as a result do not conserve energy).
The first construction of a dissipative solution of the Euler equations was given by \cite{shnirelman}. Subsequently techniques from convex integration were applied to the Euler equations to prove the existence of non-energy conserving solutions in $C^{0,1/3-}$ in \cite{isettproof}, after gradual improvement in a sequence of papers \cite{lellisinclusion,lelliscontinuous,lellisadmissibility,lellisdissipative,isettthesis,buckmaster,buckmasteralmost,buckmasteradmissible,isettendpoint,daneri} (and see references therein).

It is natural to ask whether one can consider analogues of Onsager's conjecture for related PDEs. In \cite{bardos2019} the authors consider the analogue of the first half of Onsager's conjecture and establish the universality of the Onsager exponent $\frac{1}{3}$ for conservation of entropy as well as other companion laws in a general class of conservation laws (see also \cite{bardos2019-2}). The paper \cite{beekiealpha} looked at the Euler-$\alpha$ model, while \cite{boutros} investigated several subgrid-scale $\alpha$-models of turbulence.

The goal of this paper is to establish a formulation of the Onsager conjecture for energy conservation of weak solutions to the hydrostatic Euler equations, in particular we prove results of the following type (where $B^s_{p,q} (\mathbb{T}^3)$ is a Besov space, see Definition \ref{besovdef}). The concept of a weak solution for the hydrostatic Euler equations will be made precise later.
\begin{theorem*}
Let $\mathbf{u}$ be a weak solution of the hydrostatic Euler equations (which assumes that $w \in L^2 ((0,T); L^2 (\mathbb{T}^3))$) and $(u,v) \in L^4 ( (0,T); B^\alpha_{4,\infty} (\mathbb{T}^3))$ with $\alpha > \frac{1}{2}$. Then the horizontal kinetic energy is conserved, in particular
\begin{equation*}
\lVert u(t_1, \cdot ) \rVert_{L^2}^2 + \lVert v(t_1, \cdot) \rVert^2_{L^2} = \lVert u(t_2, \cdot ) \rVert_{L^2}^2 + \lVert v(t_2, \cdot) \rVert^2_{L^2},
\end{equation*}
for almost every $t_1, t_2 \in (0,T)$.
\end{theorem*}
This theorem suggests that the analogue of the Onsager exponent (i.e. the threshold for energy conservation) for these equations is $\frac{1}{2}$. Throughout the paper the term `Onsager exponent' will denote the regularity threshold for a weak solution to conserve energy. Depending on the context, this can either refer to a Besov, H\"older or Sobolev exponent.

This result is in contrast to the Onsager exponent for the incompressible Euler equations which is $\frac{1}{3}$. See section \ref{conclusion} for a discussion of this increase of the Onsager exponent from $\frac{1}{3}$ to $\frac{1}{2}$.

The proof of this result will use techniques from \cite{duchon}. In particular we will establish an equation that describes the time evolution of the conserved quantity (referred to as the equation of local energy balance) which contains a `defect term' (which can be interpreted as an energy flux at the vanishing limit of small scales or dissipation term). This defect term captures the roughness in the weak solution which causes the anomalous dissipation of energy and will be zero if the underlying weak solution is sufficiently regular.

We observe that there are two ways to approach the issue of the decrease in regularity of $w$, as well as the nonlocality and anisotropic regularity imposed by equation \eqref{equationw}.
\begin{enumerate}
    \item Either one starts by imposing a regularity assumption on $w$ and then proceeds to derive a sufficient condition for energy conservation of weak solutions in terms of the horizontal velocities.
    \item Or one imposes sufficient regularity assumptions on the horizontal velocities such that both the vertical velocity has sufficient regularity for a weak solution to be well-defined, and the horizontal velocities are sufficiently regular for the energy to be conserved.
\end{enumerate}
We will consider both these approaches in this paper. The first approach will be considered in sections \ref{equationofenergysection}-\ref{veryweakappendix}, while the second approach will be investigated in sections \ref{sobolevappendix} and \ref{removalappendix}.

Throughout this paper, we will consider three types of weak solutions to the hydrostatic Euler equations. We will refer to these notions as type I, II or III weak solutions. For later reference, we already mention roughly what the differences between these notions are:
\begin{itemize}
    \item A type I weak solution (see Definition \ref{weaksolutiondefinition}) is the `canonical' weak solution to the hydrostatic Euler equations. For such a weak solution it is assumed that $w \in L^2 (\mathbb{T}^3 \times (0,T))$ and $u, v \in L^\infty ((0,T); L^2 (\mathbb{T}^3))$.
    \item For a type II weak solution (see Definition \ref{typeIIdefinition}) it is assumed that $w \in L^2 ( (0,T); L^2 (\mathbb{T}; \linebreak B^{-s}_{2,\infty} (\mathbb{T}^2)))$ and $u,v \in L^\infty ((0,T); L^2 (\mathbb{T}^3)) \cap L^2 ((0,T); L^2 (\mathbb{T}; B^{\sigma'}_{2,\infty} (\mathbb{T}^2)))$ for $\sigma' > s$ (here $L^2 (\mathbb{T}; B^{\sigma'}_{2,\infty} (\mathbb{T}^2))$ means $L^2 (\mathbb{T})$ regularity in the vertical direction and $B^{\sigma'}_{2,\infty} (\mathbb{T}^2)$ regularity in the horizontal directions), where $0 < s \leq \frac{1}{2}$.
    \item For a type III weak solution (see Definition \ref{veryweakdef}) we assume that $w \in L^2 ((0,T); B^{-s}_{2,\infty} (\mathbb{T}^3))$ and $u, v \in L^\infty ((0,T); L^2 (\mathbb{T}^3)) \cap L^2 ((0,T); B^{\sigma'}_{2,\infty} (\mathbb{T}^3))$ for $0 < s \leq \frac{1}{2}$ and $s < \sigma'$.
\end{itemize}
Note that a difficulty in defining weak solutions where $w$ is just a functional (i.e. type II and III weak solutions) is to make sense of the products $uw$ and $vw$ as distributions. This is done by using Bony's paradifferential calculus. Furthermore, let us point out that boundary condition \eqref{nonormalbc} is imposed by using trace theorems. In the case of type II and type III weak solutions the boundary condition will hold as a functional in a negative Besov space. We refer to Remark \ref{turbulenceremark} below for a comment on the relation of these notions of weak solution to turbulence.

We now give an outline of the paper. In section \ref{equationofenergysection} we establish the equation of local energy balance for weak solutions of type I. We then introduce a criterion that ensures that the defect term, in the local energy balance, is zero. In section \ref{regularitysection} we investigate the adequate regularity assumption necessary for the defect term to be zero, which then implies that the underlying weak solution conserves energy. In other words, the goal is to interpret the criterion derived in section \ref{equationofenergysection}. Such a sufficient condition for type I weak solutions (in terms of Besov spaces) is stated in the theorem above, which will be proven in section \ref{regularitysection}.

Inspired by the quasilinear Lipschitz condition used in proving global existence and uniqueness for the 2D Euler equations (see \cite{majda,marchioro,yudovich} for further details), we will subsequently consider a logarithmic H\"older space. We will show that functions of this kind also yield a zero defect term and solutions with such regularity conserve energy. The idea of this example is to show that the additional decay requirement of the mollified defect term (on top of the Lipschitz decay) to ensure conservation of energy is only very weak. Such a result was already proven for the case of the Euler equations by \cite{cheskidovfriedlander}.

We will also show that in order to ensure energy conservation $u$ and $v$ can have a lower H\"older exponent if $w$ is H\"older continuous. Roughly speaking, if we assume that $w (\cdot, t) \in C^{0, \beta} (\mathbb{T}^3)$ then if $u(\cdot, t), v(\cdot, t) \in C^{0, \alpha} (\mathbb{T}^3)$ with $\alpha \geq \beta > 0$ and $\alpha > \frac{1}{2} - \frac{1}{2} \beta$ there is conservation of energy. In particular, if $u,v$ and $w $ are all H\"older continuous with exponent bigger than $\frac{1}{3}$ the energy is conserved (i.e. the `original' Onsager conjecture is also true, which is consistent with the general result for conservation laws in \cite{bardos2019}).

We continue section \ref{regularitysection} by proving that it is possible to have different `horizontal' and `vertical' Onsager exponents. This is caused by the fact that the hydrostatic Euler equations are anisotropic in terms of regularity of the velocity field. We will show that if $u$ and $v$ are H\"older continuous with exponent $\alpha$ in the vertical and $\beta$ in the horizontal directions (where $2 \alpha + \beta > 2$, $\alpha > \frac{1}{3}$ and $\beta > \frac{2}{3}$) then the energy is conserved.

In section \ref{veryweakappendix} we will consider type II and type III weak solutions, namely when $w$ is allowed to be functional that belongs to a negative Besov space. We follow the same procedure as we did for type I weak solutions and prove an analogue of Onsager's conjecture for types II and III weak solutions. We again establish an equation of local energy balance with a defect term and find sufficient criteria for the vanishing of the defect term.

For type III weak solutions we show that if $w (\cdot, t) \in B^{-s}_{3,\infty} (\mathbb{T}^3)$ for $0 < s < \frac{1}{3}$ then the the condition $(u,v) \in L^3 ((0,T); B^{1/2+s/2+}_{3,\infty} (\mathbb{T}^3))$ is sufficient for the conservation of energy.
We also show that the Onsager exponent cannot be higher than $\frac{2}{3}$, for any notion of weak solution. Indeed, if $(u,v) \in L^3 ((0,T); B^{2/3+}_{3,\infty} (\mathbb{T}^3))$ then a type III weak solution of the hydrostatic Euler equations conserves energy, regardless of any regularity assumption on the vertical velocity $w$. This is because the regularity assumption on $u$ and $v$ ensures via equation \eqref{equationw} that $w$ has enough regularity to make sense of the equation and also for energy to be conserved.

The arguments for type II weak solutions are the same as for type III so therefore we only state the results without proof.

Finally, in sections \ref{sobolevappendix} and \ref{removalappendix} we formulate sufficient criteria for energy conservation solely in terms of regularity assumptions on the horizontal velocities $u$ and $v$. In section \ref{sobolevappendix} we formulate criteria in terms of Sobolev spaces, while in section \ref{removalappendix} we formulate a criterion in terms of Besov spaces.

In section \ref{conclusion} we conclude and give a plausible explanation for the increase in the Onsager exponent from $\frac{1}{3}$ for the Euler equations to $\frac{1}{2}$ for the hydrostatic Euler equations with more details. In the conclusion we also provide an overview of the different regularity criteria that ensure conservation of energy. In particular, there seems to be a `family' of Onsager conjectures for the hydrostatic Euler equations, which is due to the aforementioned anisotropic regularity of the velocity field.

In appendix \ref{besovestimatesappendix} we state a few standard estimates for Besov norms. Appendix \ref{parareview} gives an overview of the techniques from Bony's paradifferential calculus that are used in section \ref{veryweakappendix}. Finally, we note that many results in this paper could also have been proven by using the commutator estimates approach as presented in \cite{titi2018,constantin}. In appendix \ref{commutatorappendix} we give an example of a proof done by using commutator estimates, namely the proof of Proposition \ref{directionlemma}.

\section{The equation of local energy balance} \label{equationofenergysection}
In this section, we establish the equation of local energy balance. We first define a notion of weak solutions (type I) to the hydrostatic Euler equations as in \cite{chiodaroli}.
\begin{definition} \label{weaksolutiondefinition}
A type I weak solution of the hydrostatic Euler equations on the domain $\mathbb{T}^3 \times (0,T)$ is given by a velocity field $\textbf{u} = (u,v,w) : \mathbb{T}^3 \times (0,T) \rightarrow \mathbb{R}^3$ and a pressure $p : \mathbb{T}^3 \times (0,T) \rightarrow \mathbb{R}$ such that:
\begin{itemize}
    \item $\textbf{u} \in L^2 ( \mathbb{T}^3 \times (0,T))$, $u, v \in L^\infty ((0,T); L^2 (\mathbb{T}^3))$ and $p \in L^\infty ( (0,T); L^1 (\mathbb{T}^3) )$.
    \item For all $\phi_1, \phi_2 \in \mathcal{D} ( \mathbb{T}^3 \times (0,T); \mathbb{R})$ the following equations hold
    \begingroup
    \allowdisplaybreaks
    \begin{align}
    &\int_0^T \int_{\mathbb{T}^3} u \partial_t \phi_1 d x \; dt + \int_0^T \int_{\mathbb{T}^3} u \textbf{u} \cdot \nabla \phi_1 d x \; dt + \int_0^T \int_{\mathbb{T}^3} \Omega v \phi_1 d x \; dt \label{weaku}  \\
    & + \int_0^T \int_{\mathbb{T}^3} p \partial_x \phi_1 d x \; dt = 0, \nonumber \\
    &\int_0^T \int_{\mathbb{T}^3} v \partial_t \phi_2 d x \; dt + \int_0^T \int_{\mathbb{T}^3} v \textbf{u} \cdot \nabla \phi_2 d x \; dt - \int_0^T \int_{\mathbb{T}^3} \Omega u \phi_2 d x \; dt \label{weakv} \\
    &+ \int_0^T \int_{\mathbb{T}^3} p \partial_y \phi_2 d x \; dt = 0. \nonumber
    \end{align}
    \endgroup
    \item For all $\phi_3 \in \mathcal{D} (\mathbb{T}^3 \times (0,T) ; \mathbb{R})$, the following equation holds for the pressure
    \begin{equation} \label{pressureweak}
    \int_0^T \int_{\mathbb{T}^3} p \partial_z \phi_3 dx dt = 0.
    \end{equation}
    \item For all $\phi_4 \in \mathcal{D} (\mathbb{T}^3 \times (0,T) ; \mathbb{R})$ the incompressibility condition holds, i.e.
    \begin{equation} \label{incompressibleweak}
    \int_0^T \int_{\mathbb{T}^3} \textbf{u} \cdot \nabla \phi_4 d x \; dt = 0.
    \end{equation}
    \item It holds that $w(x,y,0,t) = w(x,y,1,t) = 0$ (we are considering the unit torus) for almost all $(x,y) \in \mathbb{T}^2$ and $t \in (0,T)$. This ensures that the symmetry conditions are obeyed, namely that $\mathbf{u}_H$ and $p$ are even in $z$ and that $w$ is odd in $z$. We will explain in Remark \ref{tracetypeI} why this requirement makes sense.
\end{itemize}
\end{definition}
\begin{remark} \label{tracetypeI}
If one wants to define the notion of a type I weak solution to the boundary-value problem for the hydrostatic Euler equations in the three-dimensional channel $M$, one has to ensure that the no-normal flow boundary conditions \eqref{nonormalbc} in the vertical direction are obeyed (i.e. $w(x,y,0,t) = w(x,y,L,t) = 0$). Equivalently, if the physical problem is posed on the unit torus $\mathbb{T}^3$ one has to ensure that the symmetry conditions are obeyed, which again means that the boundary conditions \eqref{nonormalbc} have to be imposed.

For type I weak solutions we know that $\mathbf{u} (\cdot, t) \in L^2 (M)$ as well as that $\nabla \cdot \mathbf{u} (\cdot, t) = 0 \in L^2 (M)$. Then by the generalised trace theorem (see \cite{tartar}) we know that $\big(\mathbf{u} \cdot \mathbf{n} (\cdot, t) \big)\lvert_{\partial M} \in H^{-1/2} (\partial M)$. In particular, on the top and bottom of the channel $\mathbf{u} \cdot \mathbf{n} = \pm w$, so the condition $w \lvert_{\partial M} = 0$ (recall that the channel has no boundary in the horizontal directions) makes sense as an equation in $H^{-1/2} (\partial M)$.

By the same reasoning, we can impose the condition $w(x,y,0,t) = 0$ on the torus. Once again, this condition makes sense as a trace. Then a weak solution of hydrostatic Euler on the torus obeying this condition can be viewed as a solution of the physical boundary-value problem on the channel because one can do an even extension in the horizontal directions and an even/odd extension in the vertical direction. For this reason, we will only consider the problem on the torus here. But the results can be proven by the same methods in the case of the channel.
\end{remark}
We will make one remark about the pressure for later reference. One can write down the following equation for the pressure
\begin{equation} \label{pressureeq}
\Delta_H p = -(\nabla_H \otimes \nabla_H) : \int_{\mathbb{T}} \bigg( \textbf{u}_H \otimes \textbf{u}_H \bigg) dz,
\end{equation}
where we recall that $\textbf{u}_H = (u,v)$ is a two-dimensional vector and $\Delta_H$ is the Laplacian in $x$ and $y$. The solution is uniquely determined by the requirement that $\int_{\mathbb{T}^2} p dx dy = 0$.

We will now show that we can weaken the regularity requirements on the test functions $\phi_1$ and $\phi_2$ in the definition of a type I weak solution (Definition \ref{weaksolutiondefinition}). We will need this result later to prove an equation of local energy balance.
\begin{lemma} \label{gentestfunctions}
Equations \eqref{weaku} and \eqref{weakv} still hold for test functions $\phi_1, \phi_2 \in W^{1,1}_0 ((0,T); L^2 (\mathbb{T}^3)) \linebreak \cap L^2 ((0,T); H^3 (\mathbb{T}^3))$.
\end{lemma}
\begin{proof}
Let $\varphi \in W^{1,1}_0 ((0,T); L^2 (\mathbb{T}^3)) \cap L^2 ((0,T); H^3 (\mathbb{T}^3))$ be arbitrary. Then there exists a sequence $\{ \varphi_n \}_{n=1}^\infty \subset \mathcal{D} (\mathbb{T}^3 \times (0,T))$ such that $\varphi_n \rightarrow \varphi$ in $W^{1,1}_0 ((0,T); L^2 (\mathbb{T}^3)) \cap L^2 ((0,T); H^3 (\mathbb{T}^3))$.

Now note that equations \eqref{weaku} and \eqref{weakv} hold for any $\varphi_n$, since they lie in $\mathcal{D} (\mathbb{T}^3 \times (0,T))$. We observe that $u \partial_t \varphi_n \rightarrow u \partial_t \varphi$ in $L^1 ((0,T); L^1 (\mathbb{T}^3))$ as $n \rightarrow \infty$ and therefore
\begin{equation*}
\int_0^T \int_{\mathbb{T}^3} u \partial_t \varphi_n d x \; dt \xrightarrow[]{n \rightarrow \infty } \int_0^T \int_{\mathbb{T}^3} u \partial_t \varphi d x \; dt.
\end{equation*}
Similarly, one can see that $u \textbf{u} \cdot \nabla \varphi_n \rightarrow u \textbf{u} \cdot \nabla \varphi$ in $L^1 ( (0,T); L^1 (\mathbb{T}^3))$ as $n \rightarrow \infty$, which means that
\begin{equation*}
\int_0^T \int_{\mathbb{T}^3} u \textbf{u} \cdot \nabla \varphi_n d x \; dt \xrightarrow[]{n \rightarrow \infty } \int_0^T \int_{\mathbb{T}^3} u \textbf{u} \cdot \nabla \varphi d x \; dt \; dt.
\end{equation*}
Recall that we have made a separate regularity assumption on the pressure ($p \in L^\infty ((0,T); \linebreak L^1 (\mathbb{T}^3))$). We observe that $p \partial_x \varphi_n \rightarrow p \partial_x \varphi$ in $L^2 ((0,T); L^1 (\mathbb{T}^3))$ as $n \rightarrow \infty$. Thus it holds that
\begin{equation*}
\int_0^T \int_{\mathbb{T}^3} p \partial_x \varphi_n d x \; dt \xrightarrow[]{n \rightarrow \infty } \int_0^T \int_{\mathbb{T}^3} p \partial_x \varphi d x \; dt \; dt.
\end{equation*}
The convergence of the other terms works exactly the same way and hence we conclude the proof.
\end{proof}
We will also fix some notation. Let $\varphi \in C_c^\infty (\mathbb{R}^3 ; \mathbb{R})$ be a radial standard $C^\infty_c$ mollifier with the property that $\int_{\mathbb{R}^3} \varphi dx = 1$. We define
\begin{equation*}
\varphi_\epsilon (x) \coloneqq \frac{1}{\epsilon^3} \varphi \bigg( \frac{x}{\epsilon} \bigg).
\end{equation*}
Moreover, we introduce the notation
\begin{equation*}
u^\epsilon \coloneqq u * \varphi_\epsilon.
\end{equation*}
Throughout the paper we will be using the Einstein summation convention.

We also recall the definition of Besov spaces. Note that there are several equivalent definitions of Besov spaces, see, e.g., \cite{leoni,bahouri} for more details.
\begin{definition} \label{besovdef}
For $1 \leq p, q \leq \infty$ and $s > 0$, we say that a measurable function $u: \mathbb{T}^n \rightarrow \mathbb{R}$ belongs to the Besov space $B^{s}_{p,q} (\mathbb{T}^n)$ if
\begin{equation*}
\lVert u \rVert_{B^{s}_{p,q} } \coloneqq \lVert u \rVert_{L^p } + \lvert u \rvert_{B^s_{p,q} } < \infty.
\end{equation*}
Here we have defined the Besov seminorm to be (for $1 \leq q < \infty$)
\begin{equation*}
\lvert u \rvert_{B^s_{p,q} } \coloneqq \bigg( \int_{\mathbb{R}^n} \lVert \Delta_h^{\lfloor s \rfloor + 1} u \rVert_{L^p }^q \frac{d h}{\lvert h \rvert^{n + s q}} \bigg)^{1/q}.
\end{equation*}
For the case $q = \infty$, we define the alternative seminorm
\begin{equation} \label{besovseminorm2}
\lvert u \rvert_{B^s_{p, \infty} } \coloneqq \sup_{h \in \mathbb{R}^n \backslash \{ 0 \} } \frac{1}{\lvert h \rvert^s} \lVert \Delta_h^{\lfloor s \rfloor + 1} u \rVert_{L^p }.
\end{equation}
Note that $\lfloor \cdot \rfloor$ has been used here to denote the integer part. We have also introduced the difference quotients through the following two equations
\begin{align*}
\Delta^1_h u(x) &\coloneqq u(x + h) - u(x), \\
\Delta^m_h u(x) &\coloneqq \Delta_h (\Delta^{m-1}_h u(x)), \quad m \geq 2,
\end{align*}
for $x \in \mathbb{T}^n$ and $h \in \mathbb{R}^n$.
\end{definition}
The equation of local energy balance will be stated in the following theorem.
\begin{theorem} \label{energyequationtheorem}
Let $\mathbf{u}$ be a type I weak solution of the hydrostatic Euler equations such that
\begin{equation} \label{L4assumption}
u, v \in L^4 ((0,T); L^4 (\mathbb{T}^3)).
\end{equation}
Then the following equation of local energy balance (which holds in the sense of distributions with the space of test functions $\mathcal{D} ( \mathbb{T}^3 \times (0,T))$) is satisfied
\begin{equation} \label{equationofenergy}
\partial_t (u^2 + v^2) + \nabla \cdot \bigg[ (u^2 + v^2) \mathbf{u} \bigg] + 2 \nabla \cdot (p \mathbf{u}) + \frac{1}{2} D (\mathbf{u}) = 0.
\end{equation}
In the above we have introduced the defect term
\begin{equation} \label{defectdefinition}
D (\mathbf{u}) (x,t) \coloneqq \lim_{\epsilon \rightarrow 0} \int_{\mathbb{R}^3} \bigg[ \nabla \varphi_\epsilon (\xi) \cdot \delta \mathbf{u} (\xi ; x, t) (\lvert\delta u (\xi ; x, t) \rvert^2 + \lvert\delta v (\xi ; x, t) \rvert^2 ) \bigg] d \xi,
\end{equation}
where
\begin{equation*}
\delta u (\xi ; x, t) \coloneqq u(x + \xi,t) - u (x,t).
\end{equation*}
The limit in equation \eqref{defectdefinition} is independent of the choice of mollifier $\varphi_\epsilon$ and the convergence is in the space $W^{-1,1} ((0,T); W^{-1,1} (\mathbb{T}^3))$.
\end{theorem}
\begin{proof}
We first mollify equations \eqref{equationu} and \eqref{equationv} in space (with $\varphi_\epsilon$) and obtain that
\begin{align}
0 &= \partial_t u^\epsilon + \nabla \cdot ( u \textbf{u})^\epsilon - \Omega v^\epsilon + \partial_x p^\epsilon, \label{mollifiedu} \\
0 &= \partial_t v^\epsilon + \nabla \cdot (v \textbf{u} )^\epsilon + \Omega u^\epsilon + \partial_y p^\epsilon. \label{mollifiedv}
\end{align}
These equations hold pointwise. We observe that $u^\epsilon, v^\epsilon \in L^\infty ((0,T); C^\infty (\mathbb{T}^3))$. We note that the terms in the system \eqref{mollifiedu}-\eqref{mollifiedv} have the following regularities: $\nabla \cdot ( u \textbf{u})^\epsilon, \nabla \cdot (v \mathbf{u})^\epsilon \in L^2 ((0,T); C^\infty (\mathbb{T}^3))$, $\Omega v^\epsilon, \Omega u^\epsilon \in L^\infty ((0,T); C^\infty (\mathbb{T}^3))$ and $\partial_x p^\epsilon, \partial_y p^\epsilon \in L^\infty ((0,T); C^\infty (\mathbb{T}^3))$ (and similarly for the terms in equation \eqref{mollifiedv}). From this we conclude that
\begin{equation*}
\partial_t u^\epsilon, \partial_t v^\epsilon \in L^2 ((0,T); C^\infty (\mathbb{T}^3)).
\end{equation*}
As a result, we obtain that
\begin{equation*}
u^\epsilon, v^\epsilon \in H^1 ((0,T); C^\infty (\mathbb{T}^3)) \cap L^\infty ((0,T); C^\infty (\mathbb{T}^3)) \subset W^{1,1} ((0,T); L^2 (\mathbb{T}^3)) \cap L^2 ((0,T); H^3 (\mathbb{T}^3)).
\end{equation*}
Therefore, we multiply $u^\epsilon$ and $v^\epsilon$ by a function $\psi \in \mathcal{D} (\mathbb{T}^3 \times (0,T); \mathbb{R})$ and consider $\psi u^\epsilon$ and $\psi v^\epsilon$ as the test functions in our weak formulation for type I weak solutions, which is allowed by Lemma \ref{gentestfunctions}. Adding the equations together gives
\begin{align}
&\int_0^T \int_{\mathbb{T}^3} \bigg[ u \partial_t (u^\epsilon \psi) + v \partial_t (v^\epsilon \psi) + u \textbf{u} \cdot \nabla (u^\epsilon \psi) + v \textbf{u} \cdot \nabla (v^\epsilon \psi) + \Omega v u^\epsilon \psi - \Omega u v^\epsilon \psi \label{integratedmollified} \\
&+ p \partial_x (u^\epsilon \psi) + p \partial_y (v^\epsilon \psi)\bigg] dx dt = 0. \nonumber
\end{align}
Now we multiply equations \eqref{mollifiedu} and \eqref{mollifiedv} by $\psi u$ and $\psi v$ respectively (they therefore hold almost everywhere), integrate them in space and time and subtract them from equation \eqref{integratedmollified}. This yields
\begin{align*}
&\int_0^T \int_{\mathbb{T}^3} \bigg[ u \partial_t (u^\epsilon \psi) - u \psi \partial_t u^\epsilon + v \partial_t (v^\epsilon \psi) - v \psi \partial_t v^\epsilon + u \textbf{u} \cdot \nabla (u^\epsilon \psi) - \psi u \nabla \cdot (u \textbf{u})^\epsilon \\
&+ v \textbf{u} \cdot \nabla (v^\epsilon \psi) - \psi v \nabla \cdot (v \textbf{u})^\epsilon  + p \partial_x (u^\epsilon \psi) - u \psi \partial_x p^\epsilon + p \partial_y (v^\epsilon \psi) - v \psi \partial_y p^\epsilon \bigg] dx dt = 0.
\end{align*}
Now we deal with the different terms in turn, we observe by the Leibniz rule that (since $u^\epsilon, v^\epsilon$ and $\psi$ are weakly differentiable in time)
\begin{align*}
&\int_0^T \int_{\mathbb{T}^3} \bigg[ u \partial_t (u^\epsilon \psi) - u \psi \partial_t u^\epsilon + v \partial_t (v^\epsilon \psi) - v \psi \partial_t v^\epsilon \bigg] dx dt = \int_0^T \int_{\mathbb{T}^3} \bigg[ (u u^\epsilon + v v^\epsilon) \partial_t \psi \bigg] dx dt \\
&= \langle -\partial_t (u u^\epsilon + v v^\epsilon), \psi \rangle,
\end{align*}
where the brackets $\langle \cdot, \cdot \rangle$ denote the distributional action. Note that we have used the fact that $\psi$ is $C^\infty$ to be able to justify the time derivative of the products $u^\epsilon \psi$ and $v^\epsilon \psi$ satisfies the Leibniz rule.

Now we look at the pressure terms, we recall that $\partial_z p^\epsilon = 0$, as well as equation \eqref{pressureweak}. Then we rewrite the pressure terms as follows
\begingroup
\allowdisplaybreaks
\begin{align*}
&\int_0^T \int_{\mathbb{T}^3} \bigg[ p \partial_x (u^\epsilon \psi) - u \psi \partial_x p^\epsilon + p \partial_y (v^\epsilon \psi) - v \psi \partial_y p^\epsilon \bigg] dx dt \\
&= \int_0^T \int_{\mathbb{T}^3} \bigg[ p \partial_x (u^\epsilon \psi) + p \partial_y (v^\epsilon \psi) + p \partial_z (w^\epsilon \psi) - u \psi \partial_x p^\epsilon - v \psi \partial_y p^\epsilon - w \psi \partial_z p^\epsilon \bigg] dx dt \\
&= \int_0^T \int_{\mathbb{T}^3} \bigg[ p \textbf{u}^\epsilon \cdot \nabla \psi + p^\epsilon \textbf{u} \cdot \nabla \psi \bigg] dx dt = \langle - \nabla \cdot ( p \textbf{u}^\epsilon + p^\epsilon \textbf{u}), \psi \rangle,
\end{align*}
\endgroup
where we have used the incompressibility $\nabla \cdot \mathbf{u}^\epsilon = 0$ condition in the above.
Finally, we look at the advective terms. We first introduce a defect term and calculate it to be
\begingroup
\allowdisplaybreaks
\begin{align*}
&D_\epsilon (\textbf{u}) \coloneqq \int_{\mathbb{R}^3} d \xi \bigg[ \nabla \varphi_\epsilon (\xi) \cdot \delta \textbf{u} (\xi ; x, t) (\lvert\delta u (\xi ; x, t) \rvert^2 + \lvert\delta v (\xi ; x, t) \rvert^2) \bigg] = - \nabla \cdot \bigg[ (u^2 + v^2) \textbf{u} \bigg]^\epsilon \\
&+ \textbf{u} \cdot \nabla (u^2 + v^2)^\epsilon + 2 u \nabla \cdot ( u \textbf{u})^\epsilon + 2 v \nabla \cdot (v \textbf{u})^\epsilon - 2 u \textbf{u} \cdot \nabla u^\epsilon - 2 v \textbf{u} \cdot \nabla v^\epsilon - (u^2 + v^2) \nabla \cdot \textbf{u}^\epsilon  \\
&= - \nabla \cdot \bigg[ (u^2 + v^2) \textbf{u} \bigg]^\epsilon +  \textbf{u} \cdot \nabla (u^2 + v^2)^\epsilon + 2 u \nabla \cdot ( u \textbf{u})^\epsilon + 2 v \nabla \cdot (v \textbf{u})^\epsilon - 2 u \textbf{u} \cdot \nabla u^\epsilon - 2 v \textbf{u} \cdot \nabla v^\epsilon.
\end{align*}
\endgroup
Subsequently, we can write the advective terms as follows
\begingroup
\allowdisplaybreaks
\begin{align*}
&\int_0^T \int_{\mathbb{T}^3} \bigg[ u \textbf{u} \cdot \nabla (u^\epsilon \psi) - \psi u \nabla \cdot (u \textbf{u})^\epsilon + v \textbf{u} \cdot \nabla (v^\epsilon \psi) - \psi v \nabla \cdot (v \textbf{u})^\epsilon \bigg] dx dt \\
&= \int_0^T \int_{\mathbb{T}^3} \bigg[ \psi u \textbf{u} \cdot \nabla u^\epsilon - \psi u \nabla \cdot (u \textbf{u})^\epsilon + \psi v \textbf{u} \cdot \nabla v^\epsilon - \psi v \nabla \cdot (v \textbf{u})^\epsilon + (u u^\epsilon + v v^\epsilon) \textbf{u} \cdot \nabla \psi \bigg] dx dt \\
&= \int_0^T \int_{\mathbb{T}^3} \bigg[ - \frac{1}{2} \psi D_\epsilon (\textbf{u}) - \frac{1}{2} \psi \nabla \cdot \bigg[ (u^2 + v^2) \textbf{u} \bigg]^\epsilon + \frac{1}{2} \psi \textbf{u} \cdot \nabla (u^2 + v^2)^\epsilon + (u u^\epsilon + v v^\epsilon) \textbf{u} \cdot \nabla \psi \bigg] dx dt \\
&= \int_0^T \int_{\mathbb{T}^3} \bigg[ - \frac{1}{2} \psi D_\epsilon (\textbf{u}) + \frac{1}{2} \bigg[ (u^2 + v^2) \textbf{u} \bigg]^\epsilon \cdot \nabla \psi - \frac{1}{2} (u^2 + v^2)^\epsilon \textbf{u} \cdot \nabla \psi + (u u^\epsilon + v v^\epsilon) \textbf{u} \cdot \nabla \psi \bigg] dx dt \\
&= \bigg\langle - \frac{1}{2} D_\epsilon (\textbf{u}) - \nabla \cdot ( (u u^\epsilon + v v^\epsilon ) \textbf{u}) + \frac{1}{2} \nabla \cdot \bigg( \big( u^2 + v^2 \big)^\epsilon \textbf{u} - \big( (u^2 + v^2) \textbf{u}\big)^\epsilon \bigg), \psi \bigg\rangle.
\end{align*}
\endgroup
Once again we have used the incompressibility condition here. Then we end up with the following distributional equation (again for all $\psi \in \mathcal{D} (\mathbb{T}^3 \times (0,T))$)
\begin{align}
&\bigg\langle \partial_t (u u^\epsilon + v v^\epsilon) + \nabla \cdot ( p \textbf{u}^\epsilon + p^\epsilon \textbf{u}) + \frac{1}{2} D_\epsilon (\textbf{u}) + \nabla \cdot ( (u u^\epsilon + v v^\epsilon ) \textbf{u}) \label{mollifiedenergyeq} \\
&+ \frac{1}{2} \nabla \cdot \bigg(  \big( (u^2 + v^2) \textbf{u}\big)^\epsilon - \big( u^2 + v^2 \big)^\epsilon \textbf{u} \bigg), \psi \bigg\rangle = 0. \nonumber
\end{align}
Now we consider the convergence of the different terms as $\epsilon \rightarrow 0$. We first observe that since $u, v\in L^\infty ((0,T); L^2 (\mathbb{T}^3))$ it holds that
\begin{equation*}
u u^\epsilon + v v^\epsilon \xrightarrow[]{\epsilon \rightarrow 0} u^2 + v^2 \quad \text{in } L^\infty ((0,T); L^1 (\mathbb{T}^3)).
\end{equation*}
Now we recall the assumption that $u, v \in L^4 ((0,T); L^4 (\mathbb{T}^3))$. By equations \eqref{equationp} and \eqref{pressureeq}, one has $p \in L^2 ((0,T); L^2 (\mathbb{T}^3))$. Therefore,
\begin{equation*}
p \textbf{u}^\epsilon + p^\epsilon \textbf{u} \xrightarrow[]{\epsilon \rightarrow 0} 2 p \textbf{u} \quad \text{in } L^1 ((0,T); L^1 (\mathbb{T}^3)).
\end{equation*}
Now again by assumption \eqref{L4assumption} it holds that
\begin{equation*}
u u^\epsilon + v v^\epsilon \xrightarrow[]{\epsilon \rightarrow 0} u^2 + v^2 \quad \text{in } L^2 ((0,T); L^2 (\mathbb{T}^3)).
\end{equation*}
In particular one has $(u u^\epsilon + v v^\epsilon) \textbf{u} \xrightarrow[]{\epsilon \rightarrow 0} (u^2 + v^2) \textbf{u} $ in $L^1 ((0,T); L^1 (\mathbb{T}^3))$. From this we conclude that
\begin{equation*}
\big( (u^2 + v^2) \textbf{u}\big)^\epsilon - \big( u^2 + v^2 \big)^\epsilon \textbf{u} \xrightarrow[]{\epsilon \rightarrow 0} 0 \quad \text{in } L^1 ((0,T); L^1 (\mathbb{T}^3)).
\end{equation*}
Finally, we need to consider the convergence of the defect term $D_\epsilon$ as $\epsilon \rightarrow 0$.
We first observe that $D_\epsilon (\textbf{u})$ is well-defined as a function in $L^1 ((0,T); L^1 (\mathbb{T}^3))$ for any $\epsilon > 0$ due to the following bound
\begin{equation*}
\lVert D_\epsilon (\textbf{u}) (\cdot, t) \rVert_{L^1 } \leq C_\epsilon \lVert \delta \textbf{u} (\cdot, t) \rVert_{L^2} (\lVert \delta u (\cdot, t) \rVert_{L^{4}}^2 + \lVert \delta v (\cdot, t) \rVert_{L^{4}}^2).
\end{equation*}
In the above $C_\epsilon$ is a constant which depends on $\epsilon > 0$. We already saw that $\delta \textbf{u} (\lvert\delta u\rvert^2 + \lvert\delta v \rvert^2) \in L^1 ((0,T); L^1 (\mathbb{T}^3))$, therefore $D_\epsilon (\textbf{u}) \in L^1 ((0,T); L^1 (\mathbb{T}^3))$ for any $\epsilon > 0$.

Regarding the limit $\epsilon \rightarrow 0$, by virtue of equation \eqref{mollifiedenergyeq} one has
\begin{equation*}
D_\epsilon (\textbf{u}) = -2\partial_t (u u^\epsilon + v v^\epsilon) -2 \nabla \cdot ( p \textbf{u}^\epsilon + p^\epsilon \textbf{u}) - 2 \nabla \cdot ( (u u^\epsilon + v v^\epsilon ) \textbf{u}) - \nabla \cdot \bigg(  \big( (u^2 + v^2) \textbf{u}\big)^\epsilon - \big( u^2 + v^2 \big)^\epsilon \textbf{u} \bigg).
\end{equation*}
In the above we have proven that the right-hand side converges in $W^{-1,1} ((0,T); W^{-1,1} (\mathbb{T}^3))$ and is independent of the choice of mollifier. We call its limit $D (\textbf{u})$.
As a result,
\begin{equation*}
D(\textbf{u}) = -2 \big( \partial_t (u^2 + v^2) + \nabla \cdot ( 2 p \textbf{u}) + \nabla \cdot ( (u^2 + v^2 ) \textbf{u}) \big),
\end{equation*}
which holds in the sense of distributions, that is
\begin{align*}
&\bigg\langle \partial_t (u^2 + v^2) + \nabla \cdot ( 2 p \textbf{u}) + \frac{1}{2} D (\textbf{u}) + \nabla \cdot ( (u^2 + v^2 ) \textbf{u}) , \psi \bigg\rangle = 0,
\end{align*}
for every test function $\psi \in \mathcal{D} ( \mathbb{T}^3 \times (0,T))$.
\end{proof}
It should be mentioned that it is possible to include the viscous terms when deriving the equation of local energy balance, but this is not needed for our purposes here. Now we introduce a criterion which ensures that the defect term, $D (\mathbf{u})$, is zero. Note that the term $D (\textbf{u})$ can be physically interpreted as an energy flux at the limit of vanishing small spatial scales, i.e. as $\epsilon \rightarrow 0$, and hence is the source of dissipation of energy in the case of not smooth enough type I weak solutions.
\begin{proposition} \label{zerodefectlemma}
Let $\mathbf{u}$ be a type I weak solution of the hydrostatic Euler equations and assume that $u,v \in L^4 ((0,T); L^4 (\mathbb{T}^3))$. Let $C \in L^1 (0,T)$ and assume that $\sigma \in L^\infty_{\mathrm{loc}} (\mathbb{R})$ is a nonnegative real-valued function with the property that $\sigma (\lvert \xi \rvert) \rightarrow 0$ as $\lvert \xi \rvert \rightarrow 0$. Suppose $\mathbf{u}$ satisfies the inequality
\begin{equation} \label{zerodefect}
\int_{\mathbb{T}^3} \lvert \delta \mathbf{u} (\xi; x,t)  \rvert \bigg(\lvert \delta u (\xi; x,t) \rvert^2 + \lvert  \delta v (\xi; x,t) \rvert^2 \bigg) dx \leq C(t) \lvert \xi \rvert \sigma ( \lvert \xi \rvert).
\end{equation}
Then $D(\mathbf{u}) = 0$.
\end{proposition}
\begin{proof}
Observe that
\begin{align*}
\lvert D_\epsilon (\textbf{u}) \rvert &= \frac{1}{2} \bigg\lvert \int_{\mathbb{R}^3} \nabla \varphi_\epsilon (\xi) \cdot \delta \textbf{u} (\xi ; x, t) \bigg[ \lvert\delta u  (\xi ; x, t) \rvert^2 + \lvert\delta v (\xi ; x, t) \rvert^2 \bigg] d \xi \bigg\rvert \\
&\leq \int_{\mathbb{R}^3} \lvert \nabla \varphi_\epsilon (\xi) \rvert \lvert \delta \textbf{u} (\xi ; x, t) \rvert \bigg[ \lvert \delta u  (\xi ; x, t) \rvert^2 + \lvert \delta v  (\xi ; x, t) \rvert^2 \bigg] d \xi.
\end{align*}
Integrating this inequality over $\mathbb{T}^3 \times (0,T)$ gives us that
\begin{align*}
\int_0^T \int_{\mathbb{T}^3} \lvert D_\epsilon (\textbf{u}) \rvert dx dt &\leq \int_0^T dt \int_{\mathbb{R}^3} \lvert \nabla \varphi_\epsilon (\xi) \rvert \int_{\mathbb{T}^3} \lvert \delta \textbf{u}  (\xi ; x, t) \rvert (\lvert \delta u (\xi ; x, t) \rvert^2 + \lvert \delta v (\xi ; x, t) \rvert^2 ) dx d \xi \\
&\leq \int_0^T C(t) dt \int_{\mathbb{R}^3} \frac{1}{\epsilon^3} \lvert \nabla_\xi \varphi (\xi / \epsilon) \rvert \lvert \xi \rvert \sigma (\lvert \xi \rvert) d \xi.
\end{align*}
Note that we are allowed to use Fubini's theorem because the integrand is absolutely integrable (as $\textbf{u} (\cdot, t) \in L^2 (\mathbb{T}^3)$ and $u (\cdot, t), v (\cdot, t) \in L^4 (\mathbb{T}^3)$).
Now by using that $\int_0^T C(t) dt < \infty$ and the change of variable $\xi = \epsilon z$, we find that
\begin{equation*}
\int_0^T \int_{\mathbb{T}^3} \lvert D_\epsilon (\textbf{u}) \rvert dx dt \lesssim  \int_{\mathbb{T}^3} \lvert \nabla_\xi \varphi (z) \rvert \lvert \xi \rvert \sigma ( \lvert \xi \rvert) d z = \int_{\mathbb{T}^3} \lvert \nabla_z \varphi (z) \rvert \lvert z \rvert \sigma (\epsilon \lvert z \rvert) dz.
\end{equation*}
Thanks to the assumption that $\sigma(\lvert \xi \rvert) \rightarrow 0$ as $\lvert \xi \rvert \rightarrow 0$, we find that $D_\epsilon (\textbf{u}) \rightarrow 0$, in $L^1 (\mathbb{T}^3 \times (0,T))$ as $\epsilon \rightarrow 0$ by the Lebesgue dominated convergence theorem. Therefore $D( \textbf{u}) = 0$.
\end{proof}
The goal of the next section is to interpret criterion \eqref{zerodefect}, i.e. by finding the biggest function space for which bound \eqref{zerodefect} holds.
\section{Proof of conservation of energy} \label{regularitysection}
We will prove several different sufficient conditions for \eqref{zerodefect} to hold. We will then prove that if condition \eqref{zerodefect} is satisfied, conservation of energy holds. In particular, one can make different regularity assumptions for $w$ due to the aforementioned anisotropy of the velocity field. In this section, we consider type I weak solutions (cf. Definition \ref{weaksolutiondefinition}), i.e. we assume that $w \in L^2 ((0,T); L^2 (\mathbb{T}^3))$.

First we will find a suitable Besov space such that condition \eqref{zerodefect} is satisfied. Later in this section we will introduce the logarithmic H\"older space in order to show that even if the solution is in a slightly weaker function space the function $\sigma$ in equation \eqref{zerodefect} is still $o(1)$ as $\lvert \xi \rvert \rightarrow 0$. We will also show that the Onsager exponent (the regularity threshold for energy conservation) can be lowered if $w $ is H\"older continuous. Finally, we will show that there can be a different Onsager exponent in the vertical direction compared to the horizontal directions.

Criterion \eqref{zerodefect} seems to dictate the condition $u (\cdot, t), v(\cdot, t) \in B^\alpha_{3, \infty} (\mathbb{T}^3)$ with $\alpha > \frac{1}{3}$, because it is possible to bound $\lvert \delta \textbf{u} \rvert$ by
\begin{equation*}
\lvert \delta \textbf{u} \rvert \leq \lvert \delta u \rvert + \lvert \delta v \rvert + \lvert \delta w \rvert.
\end{equation*}
The first two terms in this equation multiplied by the term $( \lvert\delta u \rvert^2 + \lvert\delta v\rvert^2)$ in equation \eqref{zerodefect} give us products of the form $\lvert \delta u \rvert^3, \lvert \delta v \rvert^3, \lvert \delta u \rvert \lvert \delta v \rvert^2 $ and $\lvert \delta v \rvert \lvert \delta u \rvert^2$. Because we want to treat the horizontal velocities on an equal footing, we want to impose the same conditions on the horizontal velocities. The terms in $D_\epsilon (\textbf{u})$ with only horizontal velocities should have sufficient regularity such that the function $\sigma$ in equation \eqref{zerodefect} is $o(1)$ as $\lvert \xi \rvert \rightarrow 0$. This seems to necessitate the assumption $u(\cdot, t) , v(\cdot, t) \in B^\alpha_{3, \infty} (\mathbb{T}^3)$ with $\alpha > \frac{1}{3}$.

These conditions alone do not ensure energy conservation, because in that case one needs to assume that $w (\cdot, t) \in B^\alpha_{3,\infty} (\mathbb{T}^3)$ with $\alpha > \frac{1}{3}$ as well. That would be a very strong requirement, because implicitly this imposes conditions on $u$ and $v$ via equation \eqref{equationw}. These conditions cannot be inferred directly however. Instead we need a stronger condition on $u$ and $v$, in order to not need any assumption on $w$ (except for it being in $L^2 ((0,T); L^2 (\mathbb{T}^3))$, which is required in the definition of a type I weak solution).
\subsection{Sufficient condition for energy conservation}
We now prove a proposition which gives a sufficient condition for \eqref{zerodefect} to be true, which then implies conservation of energy. This result coincides with the theorem from the introduction.
\begin{proposition}[Conservation of energy] \label{exponentlemma}
Let $\mathbf{u}$ be a type I weak solution of the hydrostatic Euler equations and assume that $u,v \in L^4( (0,T); B^\alpha_{4, \infty} (\mathbb{T}^3) )$ with $\alpha > \frac{1}{2}$. Then $D(\mathbf{u}) = 0$ (where $D(\mathbf{u}) $ was defined in equation \eqref{defectdefinition}), which implies that the weak solution conserves energy. That is,
\begin{equation} \label{energyconservationeq}
\lVert u(t_1, \cdot ) \rVert_{L^2}^2 + \lVert v(t_1, \cdot) \rVert^2_{L^2} = \lVert u(t_2, \cdot ) \rVert_{L^2}^2 + \lVert v(t_2, \cdot) \rVert^2_{L^2}, \; \text{ for a.e. $t_1, t_2 \in (0,T)$.}
\end{equation}
\end{proposition}
\begin{proof}
By the definition of a type I weak solution for the hydrostatic Euler equations (Definition \ref{weaksolutiondefinition}) it is assumed that $\textbf{u} \in L^2 ( (0,T); L^2 (\mathbb{T}^3))$. Because $u$ and $v$ are in addition assumed to be in $L^4 ((0,T); L^4 (\mathbb{T}^3))$ then by Theorem \ref{energyequationtheorem} the defect term $D (\textbf{u})$ exists and is well-defined. Moreover, the equation of local energy balance holds.
By using the additional regularity assumptions that $u(\cdot, t),v (\cdot, t) \in B^\alpha_{4, \infty} (\mathbb{T}^3)$ with $\alpha > \frac{1}{2}$ we get that
\begingroup
\allowdisplaybreaks
\begin{align*}
&\int_{\mathbb{T}^3} \lvert \delta \textbf{u} (\xi ; x, t) \rvert \bigg(\lvert  \delta u (\xi ; x, t) \rvert^2 + \lvert  \delta v (\xi ; x, t) \rvert^2 \bigg) dx \\
&\leq \lVert \delta \textbf{u} (\xi ; \cdot, t) \rVert_{L^2} ( \lVert \delta u (\xi ; \cdot, t) \rVert_{L^4}^2 + \lVert \delta v (\xi ; \cdot, t) \rVert_{L^4}^2 ) \\
&\leq 2 \lVert \textbf{u} (\cdot, t) \rVert_{L^2} \lvert \xi \rvert^{2 \alpha} \bigg( \bigg\lVert \frac{\delta u (\xi ; \cdot, t)}{\lvert \xi \rvert^\alpha} \bigg\rVert_{L^4}^2 + \bigg\lVert \frac{\delta v (\xi ; \cdot, t)}{\lvert \xi \rvert^\alpha} \bigg\rVert_{L^4}^2 \bigg) \\
&\leq 2 \lVert \textbf{u} (\cdot, t) \rVert_{L^2} \lvert \xi \rvert^{2 \alpha} \bigg( \lVert  u (\cdot, t) \rVert_{B^\alpha_{4, \infty}}^2 + \lVert  v (\cdot, t) \rVert_{B^\alpha_{4, \infty}}^2 \bigg).
\end{align*}
\endgroup
Because $2 \alpha > 1$, we can take (in the notation of Proposition \ref{zerodefectlemma})
\begin{equation*}
\sigma (\lvert \xi \rvert) \coloneqq \lvert \xi \rvert^{2\alpha - 1},
\end{equation*}
which indeed goes to zero as $\lvert \xi \rvert$ tends to zero. Finally, we observe that
\begin{equation*}
C(t) = \lVert \textbf{u} (\cdot, t) \rVert_{L^2} \bigg( \lVert  u (\cdot, t) \rVert_{B^\alpha_{4, \infty}}^2 + \lVert  v (\cdot, t) \rVert_{B^\alpha_{4, \infty}}^2 \bigg) \in L^1  (0,T) ,
\end{equation*}
since $\lVert \textbf{u} (\cdot, t) \rVert_{L^2} \in L^2 (0,T)$ and $\lVert  u (\cdot, t) \rVert_{B^\alpha_{4, \infty}}, \lVert  v (\cdot, t) \rVert_{B^\alpha_{4, \infty}} \in L^4 (0,T)$.
Therefore the conditions of Proposition \ref{zerodefectlemma} are satisfied and as a consequence it follows that $D (\textbf{u}) = 0$.

Now according to equation \eqref{equationofenergy} we have
\begin{equation*}
\partial_t (u^2 + v^2) + \nabla \cdot \bigg[ (u^2 + v^2 + 2 p) \textbf{u} \bigg] = 0,
\end{equation*}
which holds in the sense of distributions. Hence for any $\psi \in \mathcal{D} ( \mathbb{T}^3 \times (0,T); \mathbb{R})$ we have that
\begin{equation*}
\int_0^T \int_{\mathbb{T}^3} (u^2 + v^2) \partial_t \psi d x dt= - \int_0^T \int_{\mathbb{T}^3} \nabla \psi \cdot \bigg( \bigg(  u^2 + v^2+ 2p \bigg) \textbf{u} \bigg) d x dt.
\end{equation*}
Let $0 < t_1 < t_2 < T$ and choose the following test function
\begin{equation*}
\psi (t) = \int_0^t \bigg( \phi_{\epsilon} (t'-t_1) - \phi_{\epsilon} (t' - t_2) \bigg) dt'.
\end{equation*}
Here $\phi$ is a $C_c^\infty (\mathbb{R}; \mathbb{R})$ mollifier such that $\int_{\mathbb{R}} \phi (t) dt = 1$, $\text{supp} \, \phi \subset (-1,1)$ and $\phi_\epsilon (t) = \epsilon^{-1} \phi (t / \epsilon)$. Note that $\psi$ is zero for $t\in (0,t_1-\epsilon)\cup (t_2+\epsilon,T)$ (i.e. $\supp \psi \subset (0,T)$), for $\epsilon$ sufficiently small. The derivative of this test function is given by
\begin{equation*}
\partial_t \psi (t) = \phi_{\epsilon} (t-t_1) - \phi_{\epsilon} (t - t_2).
\end{equation*}
Because $\psi$ has no spatial dependence, it follows that the spatial gradient vanishes. As a result we get that
\begin{equation*}
\int_{t_1 - \epsilon}^{t_1 + \epsilon} \int_{\mathbb{T}^3} (u^2 + v^2) \phi_\epsilon (t - t_1) d x dt = \int_{t_2 - \epsilon}^{t_2 + \epsilon} \int_{\mathbb{T}^3} (u^2 + v^2) \phi_\epsilon (t - t_2) d x dt.
\end{equation*}
Then by using the Lebesgue differentiation theorem in the limit $\epsilon \rightarrow 0$ \cite{folland,zygmund} we can conclude that
\begin{equation*}
\int_{\mathbb{T}^3} \bigg( u^2 (x,t_1) + v^2 (x, t_1) \bigg) d x = \int_{\mathbb{T}^3} \bigg( u^2 (x,t_2) + v^2 (x, t_2) \bigg) d x,
\end{equation*}
for almost every $t_1,t_2 \in (0,T)$, which is equation \eqref{energyconservationeq}.
\end{proof}
It now becomes plausible why the Onsager exponent is $\frac{1}{2}$ and not $\frac{1}{3}$, the required regularity in order to ensure that the function $\sigma$ in criterion \eqref{zerodefect} is $o(1)$ has to be distributed over two terms in the product instead of three (the latter being the case for the Euler equations, see \cite{duchon} for more details). This will be discussed further in section \ref{conclusion}.
\subsection{Energy conservation for log-H\"older regularity}
Note that it is also possible to satisfy criterion \eqref{zerodefect} by imposing a different regularity condition on the solution. To this end, we introduce another function space.
\begin{definition} \label{logholderdef}
Let $U$ be a bounded and closed set. The logarithmic H\"older space $C^{0,\gamma}_{\mathrm{log}} (U)$ with $0 < \gamma < 1$ consists of all continuous functions (i.e. in $C^0 (U)$) for which the following seminorm is finite
\begin{equation} \label{seminorm}
\lvert u \rvert_{C^{0,\gamma}_{\mathrm{log}}} \coloneqq \sup_{x,y \in U, x \neq y} \frac{\lvert u (x) - u(y) \rvert}{\lvert x - y \rvert^\gamma } ( 1 + \log^- (\lvert x - y \rvert)).
\end{equation}
\end{definition}
It is straightforward to see that $C^{0,\gamma}_{\mathrm{log}} (\mathbb{T}^3) \subset C^{0, \gamma} (\mathbb{T}^3)$. One can also show that $C^{0,\gamma}_{\mathrm{log}} (\mathbb{T}^3)$ is a Banach space (see \cite{dominguez}).
We now prove that type I weak solutions with logarithmic H\"older continuity with exponent $\frac{1}{2}$ have a defect term which is zero and hence conserve energy.
\begin{proposition} \label{loglemma}
Suppose $\mathbf{u}$ is a type I weak solution of the hydrostatic Euler equations and suppose that $u,v \in L^4((0,T); C^{0,1/2}_{\mathrm{log}} (\mathbb{T}^3))$. Then $D (\mathbf{u}) = 0$, which then implies the conservation of energy  \eqref{energyconservationeq}.
\end{proposition}
\begin{proof}
The proof is almost the same as the proof of Proposition \ref{exponentlemma}. The only difference being that we now have the inequality (for $\lvert \xi \rvert < 1$)
\begin{align*}
&\lVert \delta \textbf{u} (\xi ; \cdot , t) \rVert_{L^2} ( \lVert \delta u (\xi ; \cdot, t) \rVert_{L^4}^2 + \lVert \delta v (\xi ; \cdot, t) \rVert_{L^4}^2 ) \\
&\leq 2 \lVert \textbf{u} (\cdot, t) \rVert_{L^2} \lvert \xi \rvert \frac{1}{\big(1 - \log (\lvert \xi \rvert ) \big)^2} \big( \lVert  u (\cdot, t) \rVert_{C^{0,1/2}_{\mathrm{log}}}^2 + \lVert  v (\cdot, t) \rVert_{C^{0,1/2}_{\mathrm{log}}}^2 \bigg).
\end{align*}
This means that in the notation of Proposition \ref{zerodefectlemma} we have that (again for $\lvert \xi \rvert < 1$)
\begin{equation*}
\sigma (\lvert \xi \rvert) \coloneqq \frac{1}{\big(1 - \log (\lvert \xi \rvert ) \big)^2},
\end{equation*}
which clearly goes to zero as $\lvert \xi \rvert \rightarrow 0$. Therefore the conditions of Proposition \ref{zerodefectlemma} are satisfied and the defect term is zero. The proof that the solution conserves energy is analogous to the proof of Proposition \ref{exponentlemma}.
\end{proof}
\begin{remark} \label{logholderdefect}
Note that $C^{0,1/2}_{\mathrm{log}} (\mathbb{T}^3) \not\subset B^\alpha_{4,\infty} (\mathbb{T}^3)$ with $\alpha > \frac{1}{2}$. Therefore the above proposition is not contained within Proposition \ref{exponentlemma}.
\end{remark}
\begin{remark}
We observe that the log-H\"older spaces can also be used for the original Onsager conjecture for the Euler equations. We recall that in \cite{duchon} it was proven that
\begin{equation} \label{eulercondition}
\int_{\mathbb{T}^3} \lvert \delta \textbf{u} (\xi; x,t)  \rvert^3 dx \leq C(t) \lvert \xi \rvert \sigma ( \lvert \xi \rvert),
\end{equation}
ensures that the defect term is zero (in the equation of local energy balance for a weak solution of the Euler equations) Here we used the same notation as in Proposition \ref{exponentlemma}, where $\mathbf{u}$ is a weak solution of the three-dimensional Euler equations. Now we observe that under the assumption that $u(\cdot, t) \in C^{0,1/3}_{\mathrm{log}} (\mathbb{T}^3)$ condition \eqref{eulercondition} is satisfied with the following choice for $\sigma$
\begin{equation*}
\sigma (\lvert \xi \rvert) \coloneqq \frac{1}{\big(1 - \log (\lvert \xi \rvert ) \big)^3}, \quad \lvert \xi \rvert < 1.
\end{equation*}
It is clear that $\sigma (\lvert \xi \rvert) \rightarrow 0$ as $\lvert \xi \rvert \rightarrow 0$ and therefore there is conservation of energy. Note that also in this case $\textbf{u} (\cdot, t) \in C^{0,1/3}_{\mathrm{log}} (\mathbb{T}^3)$ is a different condition from $\textbf{u} (\cdot, t) \in B^\alpha_{3, \infty} (\mathbb{T}^3)$ with $\alpha > \frac{1}{3}$. Therefore log-H\"older spaces can be applied to prove the original Onsager conjecture. It is also possible to consider even slower decay, e.g. a power of a logarithm. In addition, one can also consider logarithmic Besov spaces. A result of this type for the Euler equations was proven in \cite{cheskidovfriedlander}.

The consideration of logarithmic H\"older spaces was inspired by the proof of global existence for the 2D Euler equations with the vorticity being bounded in $L^\infty$. In that case a quasilinear Lipschitz condition is used of the form \cite{majda,marchioro,yudovich}
\begin{equation*}
\lvert u(x) - u (y) \rvert \leq K \varphi (\lvert x - y \rvert),
\end{equation*}
where we have defined
\begin{align*}
\varphi (r) \coloneqq \begin{cases}
r (1 - \log r), \quad &r < 1, \\
1, \quad &r \geq 1.
\end{cases}
\end{align*}
In order to generalise this condition we have introduced the concept of logarithmic H\"older continuity.
\end{remark}
\subsection{H\"older regularity for the vertical velocity}
Next we consider $w$ to be H\"older continuous with some exponent $\beta$. In that case the Onsager exponent (for $u$ and $v$) can be lower, as we will see in the next proposition.

\begin{proposition} \label{holderlemma}
Let $\mathbf{u} $ be a type I weak solution of the hydrostatic Euler equations and assume that $w \in L^{p_1} ((0,T); B^\beta_{p_2, \infty} (\mathbb{T}^3))$ with $\beta > 0$ and $1 \leq p_1, p_2 \leq \infty$. Moreover, assume that $u,v \in L^{q_1} ((0,T); B^\alpha_{q_2,\infty} (\mathbb{T}^3))$ with $2 p_1' \leq q_1, 2 p_2' \leq q_2$ (where $p_1', p_2'$ are the H\"older conjugates of $p_1$ and $p_2$) and $2 \alpha > 1 - \beta$ (with $\alpha \geq \beta$). Then $D (\mathbf{u}) = 0$, which implies conservation of energy \eqref{energyconservationeq}.
\end{proposition}
\begin{proof}
One can obtain the following estimates
\begingroup
\allowdisplaybreaks
\begin{align*}
&\int_{\mathbb{T}^3} \lvert \delta \textbf{u} (\xi; x,t)  \rvert \bigg(\lvert \delta u (\xi; x,t) \rvert^2 + \lvert  \delta v (\xi; x,t) \rvert^2 \bigg) dx \\
&\leq \lVert \delta \textbf{u} (\xi; \cdot, t) \rVert_{L^{p_2}} (\lVert \delta u (\xi; \cdot ,t) \rVert_{L^{2 p_2'}}^2 + \lVert \delta v (\xi; \cdot ,t) \rVert_{L^{2 p_2'}}^2) \\
&\leq \lvert \xi \rvert^{2 \alpha + \beta} \frac{\lVert \delta \textbf{u} (\xi; \cdot, t) \rVert_{L^{p_2}}}{\lvert \xi \rvert^\beta} \Bigg(\frac{\lVert \delta u (\xi; \cdot ,t) \rVert_{L^{2 p_2'}}^2}{\lvert \xi \rvert^{2 \alpha}} + \frac{\lVert \delta v (\xi; \cdot ,t) \rVert_{L^{2 p_2'}}^2}{\lvert \xi \rvert^{2 \alpha}} \Bigg) \\
&\leq \lvert \xi \rvert^{2 \alpha + \beta} \lVert \textbf{u} (\cdot, t) \rVert_{B^\beta_{p_2, \infty}} (\lVert u (\cdot, t) \rVert_{B^\alpha_{2 p_2', \infty}}^2 + \lVert v (\cdot, t) \rVert_{B^\alpha_{2 p_2', \infty}}^2 ).
\end{align*}
\endgroup
We have used the assumed Besov regularity of $w$ in estimating the term $\frac{\lVert \delta \textbf{u} (\xi; \cdot, t) \rVert_{L^{p_2}}}{\lvert \xi \rvert^\beta}$. We can now take $\sigma (\lvert \xi \rvert) \coloneqq \lvert \xi \rvert^{2 \alpha + \beta - 1}$ which goes to zero as $\lvert \xi \rvert \rightarrow 0$ since $2 \alpha + \beta > 1$ by assumption. It is also the case that $\lVert \textbf{u} (\cdot, t) \rVert_{B^\beta_{p_2, \infty}} (\lVert u (\cdot, t) \rVert_{B^\alpha_{2 p_2', \infty}}^2 + \lVert v (\cdot, t) \rVert_{B^\alpha_{2 p_2', \infty}}^2 ) \in L^1 (0,T)$. We can then apply Proposition \ref{zerodefectlemma} to conclude that $D (\textbf{u}) = 0$. By analogous reasoning to the proof of Proposition \ref{exponentlemma} we conclude that the solution conserves energy.
\end{proof}
\begin{remark}
In particular, Proposition \ref{holderlemma} implies that energy is conserved as soon as $u,v,w \in L^3 ((0,T); B^{\alpha}_{3,\infty} (\mathbb{T}^3))$ for $\alpha > \frac{1}{3}$. This means that the conservation part of the `original' Onsager conjecture still holds for the inviscid primitive equations. This is consistent with the general result for conservation laws in \cite{bardos2019}.
\end{remark}
\subsection{Anisotropic sufficient conditions for energy conservation}

As was mentioned before, the hydrostatic Euler equations are anisotropic in terms of regularity. The vertical velocity $w$ has a lower regularity than the horizontal velocities $u$ and $v$. However, one can observe from equation \eqref{equationw} that $w$ is one degree less regular in the horizontal variables but one degree more regular in the $z$-variable. We will use this observation in the next proposition to show that there are `vertical' and `horizontal' Onsager exponents. To this end, we introduce the following norm (for $0 < \alpha < \beta < 1$)
\begin{equation} \label{anisotropicbesovnorm}
\lVert f \rVert_{B^\alpha_{3,\infty} (\mathbb{T}; B^\beta_{3,\infty} (\mathbb{T}^2))} \coloneqq \lVert f \rVert_{L^3 (\mathbb{T}^3)} + \sup_{\xi \in \mathbb{R}^3 \backslash \{ 0 \}} \frac{\lVert \delta f (\cdot \; ; \xi) \rVert_{L^3 (\mathbb{T}^3)}}{\lvert \xi \rvert^\alpha} + \sup_{\xi_h \in \mathbb{R}^2 \backslash \{ 0 \}} \frac{\lVert \delta f (\cdot \; ; (\xi_h, 0)) \rVert_{L^3 (\mathbb{T}^3)}}{\lvert \xi \rvert^\beta} .
\end{equation}
A function being in the space $B^\alpha_{3,\infty} (\mathbb{T}; B^\beta_{3,\infty} (\mathbb{T}^2))$ means that it is Besov regular in the $z$-direction with exponent $\alpha$, while it has regularity $\beta$ in the horizontal directions.
\begin{proposition}[Horizontal and vertical Onsager exponents] \label{directionlemma}
Let $\mathbf{u}$ be a type I weak solution of the hydrostatic Euler equations such that $u,v \in L^3 ((0,T); B^\alpha_{3,\infty} (\mathbb{T} ; B^\beta_{3, \infty} (\mathbb{T}^2)))$ with $\beta > \frac{2}{3}, \alpha > \frac{1}{3}$ and $2 \alpha + \beta > 2$. Under these assumptions it holds that $D(\mathbf{u}) = 0$. In particular, the weak solution conserves energy \eqref{energyconservationeq}.
\end{proposition}
\begin{proof}
First we will restrict to the case $\beta > 1$, as otherwise $\alpha > \frac{1}{2}$ and then the result was proven already in Proposition \ref{exponentlemma}. By recalling equation \eqref{equationw}, we find that $\partial_z w \in L^3 ((0,T); B^{\alpha}_{3, \infty} (\mathbb{T}; B^{\beta - 1}_{3,\infty} (\mathbb{T}^2)))$. This then implies that $w \in L^3 ((0,T); B^{\alpha + 1}_{3, \infty} (\mathbb{T}; B^{\beta - 1}_{3,\infty} (\mathbb{T}^2)))$ (by using the Lebesgue differentation theorem). So $w$ has $B^{\alpha+1}_{3,\infty}$ regularity in the $z$-direction and $B^{\beta - 1}_{3,\infty}$ regularity in the horizontal directions. We take any $\xi = (\xi_h, \xi_z) \in \mathbb{R}^3$ with the horizontal increment $\xi_h \in \mathbb{R}^2$ and the vertical increment $\xi_z \in \mathbb{R}$. From this we obtain that
\begin{equation*}
\delta \textbf{u} (\xi ; x, t) = \delta \textbf{u} (\xi_z ; x + \xi_h, t) + \delta \textbf{u} (\xi_h ; x, t).
\end{equation*}
We introduce the notation $\delta \textbf{u}_h \coloneqq \delta \textbf{u} (\xi_h ; x, t)$ and $\delta \textbf{u}_z \coloneqq \delta \textbf{u} (\xi_z ; x + \xi_h, t)$. We only look at the `vertical' part of the defect term
\begin{equation*}
\int_{\mathbb{T}^3} \lvert \delta w (\xi;x,t) \rvert \bigg(\lvert \delta u (\xi; x,t) \rvert^2 + \lvert  \delta v (\xi; x,t) \rvert^2 \bigg) dx,
\end{equation*}
since the other part
\begin{equation*}
\int_{\mathbb{T}^3} (\lvert \delta u (\xi;x,t) \rvert + \lvert \delta v (\xi;x,t) \rvert ) \bigg(\lvert \delta u (\xi; x,t) \rvert^2 + \lvert  \delta v (\xi; x,t) \rvert^2 \bigg) dx
\end{equation*}
automatically satisfies the conditions of Proposition \ref{zerodefectlemma} if $\alpha, \beta > \frac{1}{3}$. For the `vertical' part of the defect term, we can derive the estimates
\begin{align*}
&\int_{\mathbb{T}^3} \lvert \delta w (\xi; x,t)  \rvert \bigg(\lvert \delta u (\xi; x,t) \rvert^2 + \lvert  \delta v (\xi; x,t) \rvert^2 \bigg) dx \\
&\leq 2 \int_{\mathbb{T}^3} (\lvert \delta w_h \rvert + \lvert \delta w_z \rvert) \bigg(\lvert \delta u_h  \rvert^2 + \lvert \delta u_z  \rvert^2 + \lvert  \delta v_h \rvert^2 + \lvert \delta v_z  \rvert^2 \bigg) dx \\
&\leq C (\lvert \xi \rvert^{\beta + 1} + \lvert \xi \rvert^{3} + \lvert \xi \rvert^{2 \alpha + \beta - 1} + \lvert \xi \rvert^{2 \alpha + 1} ) \lVert w \rVert_{B^{\alpha + 1 }_{3,\infty} (\mathbb{T} ; B^{\beta - 1}_{3,\infty} (\mathbb{T}^2))} \\
&\cdot\bigg(\lVert u \rVert_{B^{\alpha  }_{3,\infty} (\mathbb{T} ; B^{\beta }_{3,\infty} (\mathbb{T}^2))}^2 + \lVert v \rVert_{B^{\alpha  }_{3,\infty} (\mathbb{T} ; B^{\beta }_{3,\infty} (\mathbb{T}^2))}^2 \bigg).
\end{align*}
By assumption, we have that $\beta + 1 > 1$, $2 \alpha + \beta - 1 > 1$ and $2 \alpha + 1 > 1$. Moreover, we know that $C(t) = \lVert w \rVert_{B^{\alpha + 1 }_{3,\infty} (\mathbb{T} ; B^{\beta - 1}_{3,\infty} (\mathbb{T}^2))} (\lVert u \rVert_{B^{\alpha  }_{3,\infty} (\mathbb{T} ; B^{\beta }_{3,\infty} (\mathbb{T}^2))}^2 + \lVert v \rVert_{B^{\alpha  }_{3,\infty} (\mathbb{T} ; B^{\beta }_{3,\infty} (\mathbb{T}^2))}^2) \in L^1 (0,T)$. Therefore by Proposition \ref{zerodefectlemma} it follows that $D (\textbf{u}) = 0$. By a similar argument as in the proof of Proposition \ref{exponentlemma} we conclude that the solution conserves energy.
\end{proof}

Notice that this proposition is not implied by Proposition \ref{exponentlemma} in the range $\frac{1}{3} < \alpha < \frac{1}{2}$. As was mentioned before, if $\alpha > \frac{1}{2}$ then this proposition is merely a weaker result than Proposition \ref{exponentlemma}. We can conclude from Proposition \ref{directionlemma} that the Onsager exponents seem to be `direction-dependent', as can be expected from the fact that the regularity of the different components of the velocity field in the hydrostatic Euler equations is anisotropic. In other words, we can weaken the regularity requirement in the vertical direction at the expense of a stronger regularity requirement in the horizontal directions.

\section{Negative Besov regularity for the vertical velocity}
\label{veryweakappendix}

In the previous sections, in particular in Definition \ref{weaksolutiondefinition} of type I weak solutions, we assumed that $w \in L^2 ((0,T); L^2  (\mathbb{T}^3))$. It is difficult to find an equivalent regularity condition to this assumption solely in terms of $u$ and $v$. Implicity, $w \in L^2 ((0,T); L^2  (\mathbb{T}^3))$ implies some regularity of the horizontal divergence of $(u,v)$.

In order to weaken the assumptions on $w$, in this section we introduce two different notions of `very weak' solutions of the hydrostatic Euler equations. In this context, `very weak' means that either $w(\cdot, t) \in B^{-s}_{2,\infty} (\mathbb{T}^3)$ (type III weak solution) or $w (\cdot, z, t) \in B^{-s}_{2,\infty} (\mathbb{T}^2)$ (type II weak solution) for $s > 0$. Now because we need to make sense of the equation and in particular the products $u w$ and $v w$, we are obliged to make stronger regularity assumptions on $u$ and $v$ compared to Definition \ref{weaksolutiondefinition}. To this end, we will use techniques from paradifferential calculus. A brief overview of these techniques that will be used in this paper are presented in Appendix \ref{parareview}, for the convenience of the reader.

In this section, we will mainly focus on type III weak solutions. We will prove sufficient criteria for energy conservation for such solutions. The results will then be stated without proof for type II weak solutions, as they are identical in nature.

\subsection{Weak solutions with lower regularity on the vertical velocity}
In order to define the notion of a type III weak solution, we let $w$ be a functional on Besov functions. We need to prove that the products $u w$ and $ v w$ also have negative Besov regularity, for this we use the paraproduct estimates.

We will find that for type III weak solutions, the Onsager exponent seems to increase. We begin by introducing the notion of a type III weak solution.
\begin{definition} \label{veryweakdef}
We call a pair of a velocity field $\textbf{u} = (u,v,w) : \mathbb{T}^3 \times (0,T) \rightarrow \mathbb{R}^3$ and a pressure $p : \mathbb{T}^3 \times (0,T) \rightarrow \mathbb{R}$ a type III weak solution of the hydrostatic Euler equations with regularity parameter $s$ (where $s > 0$) if it satisfies the following conditions:
\begin{itemize}
    \item $w \in L^2 ((0,T); B^{-s}_{2,\infty} (\mathbb{T}^3))$, $u,v \in L^\infty ((0,T); L^2 (\mathbb{T}^3)) \cap L^{2} ((0,T); B^{\sigma'}_{2,\infty} (\mathbb{T}^3))$ for some $\sigma' > s$ and $p \in L^\infty ((0,T); L^1 (\mathbb{T}^3))$.
    \item For all $\phi_1, \phi_2 \in \mathcal{D} (\mathbb{T}^3 \times (0,T); \mathbb{R})$ the following equations hold
    \begingroup
    \allowdisplaybreaks
    \begin{align}
    &\int_0^T \int_{\mathbb{T}^3} u \partial_t \phi_1 d x \; dt + \int_0^T \langle u \textbf{u} , \nabla \phi_1 \rangle  dt + \int_0^T \int_{\mathbb{T}^3} \Omega v \phi_1 d x \; dt \label{veryweaku}  \\
    &+ \int_0^T \int_{\mathbb{T}^3} p \partial_x \phi_1 d x \; dt = 0, \nonumber \\
    &\int_0^T \int_{\mathbb{T}^3} v \partial_t \phi_2 d x \; dt + \int_0^T \langle v \textbf{u} , \nabla \phi_2 \rangle  dt - \int_0^T \int_{\mathbb{T}^3} \Omega u \phi_2 d x \; dt \label{veryweakv} \\
    &+ \int_0^T \int_{\mathbb{T}^3} p \partial_y \phi_2 d x \; dt = 0. \nonumber
    \end{align}
    \endgroup
    Note that here $\langle \cdot, \cdot \rangle$ denotes the spatial duality bracket between $\mathcal{D}' (\mathbb{T}^3)$ and $\mathcal{D} (\mathbb{T}^3)$.
    \item For all $\phi_3 \in \mathcal{D} (\mathbb{T}^3 \times (0,T) ; \mathbb{R})$ it holds that
    \begin{equation}
    \int_0^T \int_{\mathbb{T}^3} p \partial_z \phi_3 dx dt = 0.
    \end{equation}
    \item The velocity field $\mathbf{u}$ is divergence-free, i.e. for all $\phi_4 \in \mathcal{D} (\mathbb{T}^3 \times (0,T) ; \mathbb{R})$ we have that
    \begin{equation} \label{weakdivergencefree}
    \int_0^T \langle \textbf{u} , \nabla \phi_4 \rangle dt = 0.
    \end{equation}
    \item It must hold that for almost every $t \in (0,T)$ that
    \begin{equation} \label{typeIIIsymmetry}
    w(\cdot, 0,t) = w(\cdot, 1, t) = 0 \quad \text{in } B^{s-1}_{2,\infty} (\mathbb{T}^2).
    \end{equation}
\end{itemize}
\end{definition}
\begin{remark} \label{paraproductremark}
We first have to show that this definition makes sense, in particular we must be able to show that the products $u w (\cdot, t)$ and $v w (\cdot, t)$ (which appear in equations \eqref{veryweaku} and \eqref{veryweakv}) are distributions, i.e. are elements of the space $D' (\mathbb{T}^3)$. To this end, we use Bony's paradifferential calculus.
By using the regularities of $u,v$ and $w$ as stated in Definition \ref{veryweakdef} and applying Lemma~\ref{lemma:paradiff-summary} we get that $u w, v w \in L^1 ((0,T); B^{-s}_{1,\infty} (\mathbb{T}^3))$ and hence they are distributions.

In addition, we obtain that the duality brackets between $\mathcal{D}' (\mathbb{T}^3)$ and $\mathcal{D} (\mathbb{T}^3)$ are equivalent to the duality brackets between $B^{-s}_{1,\infty} (\mathbb{T}^3)$ and $B^s_{\infty,1} (\mathbb{T}^3)$, i.e.
\begin{equation*}
    \langle u \textbf{u} , \nabla \phi_1 \rangle = \langle u \textbf{u} , \nabla \phi_1 \rangle_{B^{-s}_{1,\infty} \times B^s_{\infty,1}}, \quad \langle v \textbf{u} , \nabla \phi_1 \rangle = \langle v \textbf{u} , \nabla \phi_1 \rangle_{B^{-s}_{1,\infty} \times B^s_{\infty,1}}.
\end{equation*}
In particular, we conclude that equations \eqref{equationu} and \eqref{equationv} not only hold in the sense of distributions, but in fact they hold in the space $W^{-1,1} ((0,T); B^{-s-1}_{1,\infty} (\mathbb{T}^3))$.
\end{remark}

The no-normal flow boundary condition in equation \eqref{typeIIIsymmetry} will be made sense of by using a trace theorem. In order to do so, we recall the following lemma from \cite[~Chapter 3, Lemma 1.1]{temambook} or \cite[~Lemma 1.31]{robinsonbook}.
\begin{lemma} \label{primitivelemma}
Let $X$ be a Banach space. Assume that $f,g \in L^1 ((a,b);X)$. Then the following two statements are equivalent:
\begin{enumerate}
    \item The function $f$ is equal almost everywhere to the primitive function of $g$, i.e.
    \begin{equation*}
    f (z) = \xi + \int_0^z g(z') dz', \quad \xi \in X, \text{ for almost all } z \in [a,b].
    \end{equation*}
    \item For all test functions $\phi_5 \in \mathcal{D} (a,b)$ it holds that
    \begin{equation*}
    \int_a^b f(z') \phi_5' (z') dz' = - \int_a^b g(z') \phi_5 (z') dz'.
    \end{equation*}
\end{enumerate}
In particular, both statements imply that $f$ is almost everywhere equal to a continuous function from $[a,b]$ to $X$.
\end{lemma}
By using this lemma we show that the boundary condition in \eqref{typeIIIsymmetry} is well-defined.
\begin{remark} \label{bcremark}
Firstly, by using Lemma \ref{primitivelemma} with $f = w$, $g = - \nabla_H \cdot \mathbf{u}_H$, $X = B^{s-1}_{2,\infty} (\mathbb{T}^2)$, $[a,b] = [0,1]$ and keeping $t$ fixed, we get that $w (\cdot, z, t)$ is continuous as a function of the $z$-variable (since $B^{s-1}_{2,\infty} (\mathbb{T}^2)$ is a Banach space). This subsequently means that the no-normal boundary condition \eqref{nonormalbc} for the vertical velocity is obeyed as an equation in the space $B^{s-1}_{2,\infty} (\mathbb{T}^2)$.  In particular, we are able to conclude from this that we can extend a solution on the torus to a solution in the channel.
\end{remark}
\begin{remark} \label{paraproductremark2}
To establish an equation of local energy balance, we will need to make sense of the products $u^2 w $, $v^2 w $ and $p w$ as distributions.
To this end we have to require in addition that $w\in L^3((0,T);B_{3,\infty}^{-s}(\mathbb{T}^3))$ and $u,v\in L^3((0,T);B_{3,\infty}^{\sigma'}(\mathbb{T}^3))$ for $\sigma' > s$. Again by Bony's paradifferential calculus (specifically Lemma \ref{lemma:paradiff-summary}) we find that $u^2,v^2\in L^{3/2}((0,T);B^{\sigma'}_{3/2,\infty}(\mathbb{T}^3))$. Another application of Lemma~\ref{lemma:paradiff-summary} then yields that $u^2 w, v^2 w \in L^1 ((0,T); B^{-s}_{1,\infty} (\mathbb{T}^3))$, thus they belong to $\mathcal{D}' (\mathbb{T}^3 \times (0,T); \mathbb{R})$. As a result, the action on $\mathcal{D} (\mathbb{T}^3 \times (0,T))$ makes sense and the equation of local energy balance will hold in the sense of distributions.

Regarding the pressure, we observe from equations \eqref{equationp} and \eqref{pressureeq} by using elliptic regularity results that $p \in L^{3/2} ((0,T); W^{\sigma' - \epsilon,3/2} (\mathbb{T}^3))$ (because $u^2, v^2, uv \in L^{3/2} ((0,T); B^{\sigma'}_{3/2,\infty} (\mathbb{T}^3))$, as was mentioned before), where $\epsilon > 0$ can be chosen sufficiently small such that $\sigma' - \epsilon > s$. Therefore the pressure is more regular than $L^\infty ((0,T); L^1 (\mathbb{T}^3))$, as was assumed in the definition of a type III weak solution. Then by proceeding as before we can show that $p \textbf{u} \in L^1 ((0,T); B^{-s}_{1,\infty} (\mathbb{T}^3))$, which will be needed to establish the equation of local energy balance.
\end{remark}
\begin{remark} \label{paraproductremark3}
Note that in the case $s > \frac{1}{2}$ the regularity requirement $w(\cdot, t) \in B^{-s}_{2,\infty} (\mathbb{T}^3)$ already follows from the regularity requirement on $u$ and $v$. Indeed from $u (\cdot, t),v (\cdot, t) \in B^{\sigma'}_{2,\infty} (\mathbb{T}^3)$ (for some $\sigma' > s$) and equation \eqref{weakdivergencefree} we deduce that $\partial_z w (\cdot, t) \in L^2 (\mathbb{T}; B^{\sigma'-1}_{2,\infty} (\mathbb{T}^2))$. Then we find that $w (\cdot, t) \in L^2 (\mathbb{T} ;B^{\sigma'-1}_{2,\infty} (\mathbb{T}^2))$. This can be seen by applying Lemma \ref{primitivelemma} and using the no-normal flow boundary condition imposed in Definition \ref{veryweakdef}.

Then we observe that $w (\cdot, t) \in B^{\sigma'-1}_{2,\infty} (\mathbb{T}^3) \subset B^{-s}_{2,\infty} (\mathbb{T}^3)$. The inclusion holds because $\sigma'-1>s-1>-\frac{1}{2}>-s$. Equivalently, if $\mathbf{u}$ is a type III weak solution with regularity parameter $s \in (\frac{1}{2},1)$, then $\mathbf{u}$ is in fact a type III weak solution with smaller regularity parameter $1-s$.

Naturally, if $s \geq 1$ then $\partial_z w (\cdot, t)$ will lie in $L^2 (\mathbb{T}^3)$. Then by applying Lemma \ref{primitivelemma} we find that $w (\cdot, t) \in L^2 (\mathbb{T}^3)$, which means that we are considering a type I weak solution.

\end{remark}
Now we show that we can expand the space of test functions, as it was done in Lemma \ref{gentestfunctions}.
\begin{lemma} \label{genweaktestfunctions}
Let $\mathbf{u}$ be a type III weak solution of the hydrostatic Euler equations with regularity parameter $s$, such that $w\in L^3((0,T);B_{3,\infty}^{-s}(\mathbb{T}^3))$ and $u,v\in L^3((0,T);B_{3,\infty}^{\sigma'}(\mathbb{T}^3))$. Then the weak formulation in equations \eqref{veryweaku} and \eqref{veryweakv} still holds for test functions in $\phi_1, \phi_2 \in W^{1,1}_0 ((0,T); L^2 (\mathbb{T}^3)) \cap L^{3} ((0,T); H^{3 + s} (\mathbb{T}^3)) $.
\end{lemma}
\begin{proof}
The proof works the same way as the proof for Lemma \ref{gentestfunctions}. We take an arbitrary $\varphi \in W^{1,1}_0 ((0,T); L^2 (\mathbb{T}^3)) \cap L^{3} ((0,T); H^{3 + s} (\mathbb{T}^3))$. Then we know that there exists a sequence $\{ \varphi_n \}_{n=1}^\infty \subset \mathcal{D} (\mathbb{T}^3 \times (0,T))$ such that $\varphi_n \rightarrow \varphi$ in $W^{1,1}_0 ((0,T); L^2 (\mathbb{T}^3)) \cap L^{3} ((0,T); H^{3 + s} (\mathbb{T}^3))$ as $n \rightarrow \infty$.

The only convergence that we need to show that differs from the proof of Lemma \ref{gentestfunctions} is for the integral $\int_0^T \langle u \textbf{u} , \nabla \varphi_n \rangle_{B^{-s}_{1,\infty} \times B^s_{\infty,1} } dt$ (and the same integral for $v$). We observe that $\nabla \varphi_n \xrightarrow[]{n \rightarrow \infty} \nabla \varphi$ in $L^3 ((0,T); H^{2 + s} (\mathbb{T}^3)) \subset L^{3} ((0,T); B^s_{\infty,1} (\mathbb{T}^3))$. This means that we can prove the convergence
\begin{align*}
&\bigg\lvert \int_0^T \langle u \textbf{u} , \nabla \varphi_n \rangle_{B^{-s}_{1,\infty} \times B^s_{\infty,1} } dt - \int_0^T \langle u \textbf{u} , \nabla \varphi \rangle_{B^{-s}_{1,\infty} \times B^s_{\infty,1} } dt \bigg\rvert \\
&= \bigg\lvert \int_0^T \langle u \textbf{u} , \nabla \varphi - \nabla \varphi_n \rangle_{B^{-s}_{1,\infty} \times B^s_{\infty,1} } dt \bigg\rvert \\
&\leq \int_0^T \big( \lVert u \textbf{u} \rVert_{B^{-s}_{1,\infty}} \lVert \nabla \varphi - \nabla \varphi_n \rVert_{B^{s}_{\infty,1}} \big) dt \leq \lVert u \textbf{u} \rVert_{L^{3/2}_t ((B^{-s}_{1,\infty})_x)} \lVert  \varphi -  \varphi_n \rVert_{L^{3}_t ((B^{s+1}_{\infty,1})_x)} \\
&\xrightarrow[]{n \rightarrow \infty} 0.
\end{align*}
Here we have used that $u \mathbf{u} \in L^{3/2} ( (0,T); B^{-s}_{1,\infty} (\mathbb{T}^3))$ which follows from the regularity assumptions on $u,v,w$ and Lemma~\ref{lemma:paradiff-summary}. Therefore the final inequality follows, which concludes the proof.
\end{proof}

\subsection{The equation of local energy balance}
Now that we have introduced the notion of a type III weak solution and extended the space of test functions, we can establish an equation of local energy balance.
\begin{theorem} \label{weakenergyequation}
Let $\mathbf{u}$ be a type III weak solution of the hydrostatic Euler equations with regularity parameter $s$, such that $w\in L^3((0,T);B_{3,\infty}^{-s}(\mathbb{T}^3))$ and $u,v\in L^3((0,T);B_{3,\infty}^{\sigma'}(\mathbb{T}^3))$. Then for all $\psi \in \mathcal{D} (\mathbb{T}^3 \times (0,T))$ the following equation of local energy balance is satisfied
\begin{align}
&\int_0^T \bigg[ \int_{\mathbb{T}^3} \bigg( u^2 \partial_t \psi + v^2 \partial_t \psi  \bigg) dx + \bigg\langle  (u^2 + v^2 + 2p) \mathbf{u}, \nabla \psi \bigg\rangle_{B^{-s}_{1,\infty} \times B^s_{\infty,1}} \nonumber \\
&- \frac{1}{2} \bigg\langle  D(\mathbf{u}), \psi \bigg\rangle_{B^{-s}_{1,\infty} \times B^s_{\infty,1} }  \bigg] dt = 0. \label{energybalance}
\end{align}
Here $D (\mathbf{u})$ denotes the element in $W^{-1,1} ((0,T); B^{-s-1}_{1,\infty} (\mathbb{T}^3))$ given by
\begin{equation*}
D (\mathbf{u}) \coloneqq \lim_{\epsilon \rightarrow 0} \int_{\mathbb{R}^3} d \xi \bigg[ \nabla \varphi_\epsilon (\xi) \cdot \delta \mathbf{u} (\xi ; x, t) (\lvert\delta u (\xi ; x, t) \rvert^2 + \lvert\delta v (\xi ; x, t) \rvert^2) \bigg].
\end{equation*}
The limit $D (\mathbf{u})$ is independent of the choice of mollifier.
\end{theorem}
\begin{remark}
Notice that in particular, since $L^\infty ((0,T); B^{s+1} (\mathbb{T}^3)) \subset \mathcal{D} (\mathbb{T}^3 \times (0,T); \mathbb{R})$, equation \eqref{energybalance} holds in the sense of distributions (the duality brackets in this equation are equivalent to the distributional action between $\mathcal{D}' (\mathbb{T}^3 \times (0,T))$ and $\mathcal{D} (\mathbb{T}^3 \times (0,T))$). However, we show in fact that the equation holds in the space $W^{-1,1} ((0,T); B^{-s-1}_{1,\infty} (\mathbb{T}^3))$.
\end{remark}
\begin{proof}
Throughout the proof we will write duality brackets with explicit subscripts to make clear in which spaces the elements lie as to make the proof more understandable. Any duality bracket in this proof is equivalent to a distributional duality bracket, as $\psi$ is a test function.

We first mollify equations \eqref{equationu} and \eqref{equationv} in space with $\varphi_\epsilon$ and obtain
\begin{align}
0 &= \partial_t u^\epsilon + \nabla \cdot ( u \textbf{u})^\epsilon - \Omega v^\epsilon + \partial_x p^\epsilon, \label{mollifiedu2} \\
0 &= \partial_t v^\epsilon + \nabla \cdot (v \textbf{u} )^\epsilon + \Omega u^\epsilon + \partial_y p^\epsilon. \label{mollifiedv2}
\end{align}
It holds that $u^\epsilon, v^\epsilon \in L^\infty ((0,T); C^\infty (\mathbb{T}^3))$. We observe that $( u \textbf{u})^\epsilon,  (v \textbf{u})^\epsilon \in L^1 ((0,T); C^\infty (\mathbb{T}^3))$ and $p^\epsilon \in L^\infty ((0,T); C^\infty (\mathbb{T}^3))$, from this we can conclude that
$\partial_t u^\epsilon, \partial_t v^\epsilon \in L^1 ((0,T) ; C^\infty (\mathbb{T}^3))$, which implies that $u^\epsilon, v^\epsilon \in W^{1,1} ((0,T) ; C^\infty (\mathbb{T}^3)) \cap L^\infty ((0,T); C^\infty (\mathbb{T}^3))$.

As a result, we are able to apply Lemma \ref{genweaktestfunctions} and use $\psi u^\epsilon$ and $\psi v^\epsilon $ as test functions in the weak formulation in Definition \ref{veryweakdef} (for $\psi \in \mathcal{D} (\mathbb{T}^3 \times (0,T); \mathbb{R})$). Adding the equations for $u$ and $v$ together gives that
\begin{align*}
&\int_0^T \bigg[ \int_{\mathbb{T}^3}  \bigg( u \partial_t (u^\epsilon \psi) + v \partial_t (v^\epsilon \psi) + \Omega v u^\epsilon \psi - \Omega u v^\epsilon \psi + p \partial_x (u^\epsilon \psi) + p \partial_y (v^\epsilon \psi)\bigg) dx \\
&+ \langle u \textbf{u} , \nabla (u^\epsilon \psi) \rangle_{B^{-s}_{1,\infty} \times B^s_{\infty,1} } + \langle v \textbf{u} , \nabla (v^\epsilon \psi) \rangle_{B^{-s}_{1,\infty} \times B^s_{\infty,1} } \bigg] dt = 0.
\end{align*}
By multiplying equations \eqref{mollifiedu2} and \eqref{mollifiedv2} by $u \psi$ respectively $v \psi$ and subtracting them from the previous equation, we get that
\begingroup
\allowdisplaybreaks
\begin{align*}
&\int_0^T \bigg[ \int_{\mathbb{T}^3} \bigg( u \partial_t (u^\epsilon \psi) - u \psi \partial_t u^\epsilon + v \partial_t (v^\epsilon \psi) - v \psi \partial_t v^\epsilon + p \partial_x (u^\epsilon \psi) - u \psi \partial_x p^\epsilon  \\
&+ p \partial_y (v^\epsilon \psi) - v \psi \partial_y p^\epsilon - \psi u \nabla \cdot (u \textbf{u})^\epsilon - \psi v \nabla \cdot (v \textbf{u})^\epsilon \bigg) dx + \langle u \textbf{u} , \nabla (u^\epsilon \psi) \rangle_{B^{-s}_{1,\infty} \times B^s_{\infty,1} }  \\
&+ \langle v \textbf{u} , \nabla (v^\epsilon \psi) \rangle_{B^{-s}_{1,\infty} \times B^s_{\infty,1} }  \bigg] dt  \\
&= \int_0^T \bigg[ \int_{\mathbb{T}^3} \bigg( u \partial_t (u^\epsilon \psi) - u \psi \partial_t u^\epsilon + v \partial_t (v^\epsilon \psi) - v \psi \partial_t v^\epsilon + p \partial_x (u^\epsilon \psi) - u \psi \partial_x p^\epsilon  \\
&+ p \partial_y (v^\epsilon \psi) - v \psi \partial_y p^\epsilon \bigg) dx + \bigg\langle u \textbf{u} \cdot \nabla u^\epsilon - u \nabla \cdot (u \textbf{u})^\epsilon + v \textbf{u} \cdot \nabla v^\epsilon - v \nabla \cdot (v \textbf{u})^\epsilon, \psi \bigg\rangle_{B^{-s}_{1,\infty} \times B^s_{\infty,1} }  \\
&+ \bigg\langle  (u u^\epsilon + v v^\epsilon) \textbf{u}, \nabla \psi \bigg\rangle_{B^{-s}_{1,\infty} \times B^s_{\infty,1}} \bigg] dt = 0.
\end{align*}
\endgroup
We will mainly discuss the advective terms here, as the derivation of the equation of energy for the other terms proceeds in the same way as Theorem \ref{energyequationtheorem}.  To handle them, we introduce a defect term
\begingroup
\allowdisplaybreaks
\begin{align*}
D_\epsilon (\textbf{u}) &\coloneqq \int_{\mathbb{R}^3} d \xi \bigg[ \nabla \varphi_\epsilon (\xi) \cdot \delta \textbf{u} (\xi ; x, t) (\lvert\delta u (\xi ; x, t) \rvert^2 + \lvert\delta v (\xi ; x, t) \rvert^2) \bigg] \\
&= - \nabla \cdot \bigg[ (u^2 + v^2) \textbf{u} \bigg]^\epsilon +  \textbf{u} \cdot \nabla (u^2 + v^2)^\epsilon + 2 u \nabla \cdot ( u \textbf{u})^\epsilon + 2 v \nabla \cdot (v \textbf{u})^\epsilon \\
&- 2 u \textbf{u} \cdot \nabla u^\epsilon - 2 v \textbf{u} \cdot \nabla v^\epsilon.
\end{align*}
\endgroup
Notice that the equation holds in $L^1 ((0,T); B^{-s}_{1,\infty} (\mathbb{T}^3))$, in particular in $\mathcal{D}' (\mathbb{T}^3 \times (0,T); \mathbb{R}) $.

Now we can write the advective terms as follows
\begingroup
\allowdisplaybreaks
\begin{align*}
&\int_0^T \bigg\langle u \textbf{u} \cdot \nabla u^\epsilon - u \nabla \cdot (u \textbf{u})^\epsilon + v \textbf{u} \cdot \nabla v^\epsilon - v \nabla \cdot (v \textbf{u})^\epsilon, \psi \bigg\rangle_{B^{-s}_{1,\infty} \times B^s_{\infty,1}} dt \\
&= \int_0^T \bigg\langle - \frac{1}{2 } D_\epsilon (\textbf{u}) - \nabla \cdot \bigg[ (u^2 + v^2) \textbf{u} \bigg]^\epsilon + \textbf{u} \cdot \nabla (u^2 + v^2)^\epsilon, \psi \bigg\rangle_{B^{-s}_{1,\infty} \times B^s_{\infty,1}} dt \\
&= \int_0^T \bigg\langle - \frac{1}{2 } D_\epsilon (\textbf{u}) , \psi \bigg\rangle_{B^{-s}_{1,\infty} \times B^s_{\infty,1}} + \bigg\langle \bigg[ (u^2 + v^2) \textbf{u} \bigg]^\epsilon - \textbf{u}  (u^2 + v^2)^\epsilon, \nabla \psi \bigg\rangle_{B^{-s}_{1,\infty} \times B^s_{\infty,1}} dt.
\end{align*}
\endgroup
We observe that $\bigg[ (u^2 + v^2) \textbf{u} \bigg]^\epsilon - \textbf{u}  (u^2 + v^2)^\epsilon \rightarrow 0$ in $B^{-s}_{1,\infty} (\mathbb{T}^3)$ by the paraproduct estimates (see Remark \ref{paraproductremark2} and Appendix \ref{parareview} for more details). In particular, we have that (again for some $\sigma' > s$, by the assumptions on $u,v,w$ and Lemma~\ref{lemma:paradiff-summary})
\begingroup
\allowdisplaybreaks
\begin{align*}
&\bigg\lVert \bigg[ (u^2 + v^2) \textbf{u} \bigg]^\epsilon - \textbf{u}  (u^2 + v^2)^\epsilon \bigg\rVert_{B^{-s}_{1,\infty}} \\
&\leq \bigg\lVert \bigg[ (u^2 + v^2) \textbf{u} \bigg]^\epsilon - \textbf{u}  (u^2 + v^2) \bigg\rVert_{B^{-s}_{1,\infty}} + \bigg\lVert  (u^2 + v^2) \textbf{u} - \textbf{u}  (u^2 + v^2)^\epsilon \bigg\rVert_{B^{-s}_{1,\infty}} \\
&\leq \bigg\lVert \bigg[ (u^2 + v^2) \textbf{u} \bigg]^\epsilon - \textbf{u}  (u^2 + v^2) \bigg\rVert_{B^{-s}_{1,\infty}} + \lVert  (u^2 + v^2) - (u^2 + v^2)^\epsilon \rVert_{B^{\sigma'}_{3/2,\infty}} \lVert \textbf{u} \rVert_{B^{-s}_{3,\infty}} \xrightarrow[]{\epsilon \rightarrow 0} 0.
\end{align*}
\endgroup
Finally, we need to prove that the defect term makes sense. First we claim that
\begin{equation} \label{integralclaim}
\bigg\langle  D_\epsilon (\textbf{u}) , \psi \bigg\rangle_{B^{-s}_{1,\infty} \times B^s_{\infty,1}} = \int_{\mathbb{R}^3}  \bigg[ \nabla \varphi_\epsilon (\xi) \langle \delta \textbf{u}  (\lvert \delta u \rvert^2  + \lvert \delta v \rvert^2  ), \psi \rangle_{B^{-s}_{1,\infty} \times B^s_{\infty,1} } \bigg] d \xi .
\end{equation}
This can be shown as follows. For the mollification (in space) $\mathbf{u}^\eta$ of $\mathbf{u}$, we have by Fubini's theorem
\begin{align*}
&\bigg\langle  \int_{\mathbb{R}^3} d \xi \; \nabla \varphi_\epsilon (\xi) \cdot \delta \textbf{u}^\eta  (\lvert\delta u\rvert^2  + \lvert\delta v\rvert^2 ) , \psi \bigg\rangle_{B^{-s}_{1,\infty} \times B^s_{\infty,1}} \\
&= \int_{\mathbb{R}^3} d \xi \bigg[ \nabla \varphi_\epsilon (\xi) \langle \delta \textbf{u}^\eta  (\lvert \delta u \rvert^2 + \lvert \delta v \rvert^2  ), \psi \rangle_{B^{-s}_{1,\infty} \times B^s_{\infty,1} } \bigg].
\end{align*}
Now we observe that
\begingroup
\allowdisplaybreaks
\begin{align*}
&-\int_{\mathbb{R}^3} d \xi \bigg[ \nabla \varphi_\epsilon (\xi) \langle \delta \textbf{u}^\eta (\lvert \delta u \rvert^2 + \lvert \delta v \rvert^2 ), \psi \rangle_{B^{-s}_{1,\infty} \times B^s_{\infty,1} } \bigg] \\
&+ \int_{\mathbb{R}^3} d \xi \bigg[ \nabla \varphi_\epsilon (\xi) \langle \delta \textbf{u}  (\lvert \delta u \rvert^2  + \lvert \delta v \rvert^2 ), \psi \rangle_{B^{-s}_{1,\infty} \times B^s_{\infty,1} } \bigg] \\
&= \int_{\mathbb{R}^3} d \xi \bigg[ \nabla \varphi_\epsilon (\xi) \langle \delta (\textbf{u} - \textbf{u}^\eta)  (\lvert \delta u \rvert^2  + \lvert \delta v \rvert^2  ), \psi \rangle_{B^{-s}_{1,\infty} \times B^s_{\infty,1} } \bigg] \\
&\leq C \lVert \textbf{u} - \textbf{u}^\eta \rVert_{B^{-s}_{3,\infty}} (\lVert u \rVert_{B^{\sigma'}_{3,\infty}}^2 + \lVert v \rVert_{B^{\sigma'}_{3,\infty}}^2) \lVert \psi \rVert_{B^s_{\infty,1}} dt \xrightarrow[]{\eta \rightarrow 0} 0.
\end{align*}
\endgroup
Moreover, it is straightforward to verify that
\begingroup
\allowdisplaybreaks
\begin{align*}
&\bigg\langle  D_\epsilon (\textbf{u}) - \int_{\mathbb{R}^3} d \xi \; \nabla \varphi_\epsilon (\xi) \cdot \delta \textbf{u}^\eta (\lvert\delta u\rvert^2 + \lvert\delta v\rvert^2) , \psi \bigg\rangle_{B^{-s}_{1,\infty} \times B^s_{\infty,1} }  \\
&\leq C \lVert \textbf{u} - \textbf{u}^\eta \rVert_{B^{-s}_{3,\infty}} (\lVert u \rVert_{B^{\sigma'}_{3,\infty}}^2 + \lVert v \rVert_{B^{\sigma'}_{3,\infty}}^2) \lVert \psi \rVert_{B^s_{\infty,1}} \xrightarrow[]{\eta \rightarrow 0} 0.
\end{align*}
\endgroup
Therefore claim \eqref{integralclaim} is true. Hence we have
\begin{align*}
\langle  D_\epsilon (\textbf{u}) , \psi \rangle_{B^{-s}_{1,\infty} \times B^s_{\infty,1}} &= \int_{\mathbb{R}^3} d \xi \bigg[ \nabla \varphi_\epsilon (\xi) \langle \delta \textbf{u} (\lvert \delta u \rvert^2 + \lvert \delta v \rvert^2 ), \psi \rangle_{B^{-s}_{1,\infty} \times B^s_{\infty,1} } \bigg] \\
&\leq C \lVert \psi \rVert_{B^s_{\infty,1}} \lVert \textbf{u} \rVert_{B^{-s}_{3,\infty}} (\lVert u \rVert_{B^{\sigma'}_{3,\infty}}^2 + \lVert v \rVert_{B^{\sigma'}_{3,\infty}}^2) < \infty.
\end{align*}
In particular, $D_\epsilon (\textbf{u})$ makes sense as an element in $L^1 ((0,T);B^{-s}_{1,\infty} (\mathbb{T}^3))$. The convergence as $\epsilon \rightarrow 0$ follows by looking at the equation we established (for fixed $\epsilon > 0$)
\begin{equation*}
D_\epsilon (\textbf{u}) = -2\partial_t (u u^\epsilon + v v^\epsilon) -2 \nabla \cdot ( p \textbf{u}^\epsilon + p^\epsilon \textbf{u}) - 2 \nabla \cdot ( (u u^\epsilon + v v^\epsilon ) \textbf{u}) - \nabla \cdot \bigg(  \big( (u^2 + v^2) \textbf{u}\big)^\epsilon - \big( u^2 + v^2 \big)^\epsilon \textbf{u} \bigg).
\end{equation*}
As $\epsilon \rightarrow 0$ the right-hand side converges in $W^{-1,1} ((0,T); B^{-1-s}_{1,\infty} (\mathbb{T}^3))$ which means that $D_\epsilon (\mathbf{u})$ converges in the same space. Since the limit of the right-hand side is independent of the choice of mollifier, hence the limit $D (\mathbf{u})$ is independent of the choice of mollifier.
Therefore we conclude that the derived equation of local energy balance holds in the sense of distributions.
\end{proof}

\begin{remark}
    In order to make sense of the hydrostatic Euler equations for a type III weak solution, the assumptions $w\in L^2 ((0,T); B^{-s}_{2,\infty} (\mathbb{T}^3))$ and $u,v \in L^2 ((0,T); B^{\sigma'}_{2,\infty} (\mathbb{T}^3))$ for some $\sigma' > s$ are sufficient. To derive an equation of local energy balance however, the assumption that $w\in L^3 ((0,T); B^{-s}_{3,\infty} (\mathbb{T}^3))$ and $u, v \in L^{3} ((0,T); B^{\sigma'}_{3,\infty} ( \mathbb{T}^3))$ is necessary. Alternatively we could require $w\in L^2 ((0,T); B^{-s}_{2,\infty} (\mathbb{T}^3))$ and $u, v \in L^{4} ((0,T); B^{\sigma'}_{4,\infty} ( \mathbb{T}^3))$ in order to establish an equation of local energy balance. Propositions \ref{weakzerodefect}, \ref{weakdefectzero}, \ref{upperbound} and Theorem~\ref{veryweakconservation} then hold with straightforward modifications.
\end{remark}

\subsection{Proof of conservation of energy}
Now we prove an estimate on the defect term like in Proposition \ref{zerodefectlemma}.
\begin{proposition} \label{weakzerodefect}
Let $\mathbf{u}$ be a type III weak solution of the hydrostatic Euler equations with regularity parameter $s$ such that $w\in L^3((0,T);B_{3,\infty}^{-s}(\mathbb{T}^3))$ and $u,v\in L^3((0,T);B_{3,\infty}^{\sigma'}(\mathbb{T}^3))$. Assume that for all $\psi \in \mathcal{D} ( \mathbb{T}^3 \times (0,T))$ it holds that
\begin{equation} \label{weakzerodefecteq}
\bigg\lvert \langle \delta \mathbf{u} (\lvert \delta u \rvert^2 + \lvert \delta v \rvert^2 ), \psi \rangle_{B^{-s}_{1,\infty} \times B^s_{\infty,1} } \bigg\rvert \leq C(t) \lvert \xi \rvert \sigma (\lvert \xi \rvert) \lVert \psi \rVert_{B^s_{\infty,1}},
\end{equation}
where $C \in L^1 (0,T)$ and $\sigma \in L^\infty_{\mathrm{loc}} (\mathbb{R})$ with the property that $\sigma (\lvert \xi \rvert) \rightarrow 0 $ as $\lvert \xi \rvert \rightarrow 0$. Then $D (\mathbf{u}) = 0$, which implies that energy is conserved.
\end{proposition}
\begin{proof}
Under this assumption, we estimate the defect term as follows (by using identity \eqref{integralclaim})
\begingroup
\allowdisplaybreaks
\begin{align*}
\bigg\lvert \int_0^T \bigg\langle - \frac{1}{2 } D_\epsilon (\textbf{u}) , \psi \bigg\rangle_{B^{-s}_{1,\infty} \times B^s_{\infty,1}} dt \bigg\rvert &=\bigg\lvert \int_0^T \int_{\mathbb{R}^3} \nabla \varphi_\epsilon (\xi) \langle \delta \textbf{u} (\lvert \delta u \rvert^2 + \lvert \delta v \rvert^2 ), \psi \rangle_{B^{-s}_{1,\infty} \times B^s_{\infty,1} } d \xi dt \bigg\rvert \\
&\leq \int_0^T \int_{\mathbb{R}^3} \lvert \nabla \varphi_\epsilon (\xi) \rvert \bigg\lvert \langle \delta \textbf{u} (\lvert \delta u \rvert^2 + \lvert \delta v \rvert^2 ), \psi \rangle_{B^{-s}_{1,\infty} \times B^s_{\infty,1} } \bigg\rvert d \xi dt \\
&\leq \int_0^T C(t) \lVert \psi \rVert_{B^s_{\infty,1}} dt \int_{\mathbb{R}^3} \lvert \nabla \varphi_\epsilon (\xi) \rvert \lvert \xi \rvert \sigma (\lvert \xi \rvert) d \xi \\
&\leq \int_0^T C(t) \lVert \psi \rVert_{B^s_{\infty,1}} dt \int_{\mathbb{R}^3} \lvert \nabla_z \varphi (z) \rvert \lvert z \rvert \sigma ( \epsilon \lvert z \rvert) d z,
\end{align*}
\endgroup
where in the last line we made a change of variable $\xi = \epsilon z$ (like was done in Proposition \ref{zerodefectlemma}).
Using the dominated convergence theorem and the assumption on $\sigma$ we deduce that the right-hand side converges to 0 as $\epsilon \rightarrow 0$. Therefore we conclude that $D (\textbf{u}) = 0$.
\end{proof}
Next we give sufficient conditions that ensure that the defect term is zero, in particular such that condition \ref{weakzerodefecteq} is satisfied.
\begin{proposition} \label{weakdefectzero}
Assume that $\mathbf{u}$ is a type III weak solution of the hydrostatic Euler equations with regularity parameter $s$, and that $w\in L^3((0,T);B^{-s}_{3,\infty}(\mathbb{T}^3))$ and $u, v \in L^{3} ((0,T); B^\alpha_{3, \infty} (\mathbb{T}^3))$ with $\alpha > \frac{1}{2} + \frac{s}{2}$. Then $D (\mathbf{u}) = 0$ in the sense of distributions. This then implies that the weak solution conserves energy.
\end{proposition}

In order to prove Proposition~\ref{weakdefectzero} we need the following lemma.

\begin{lemma} \label{lemma:besov-delta}
    Let $1\leq p\leq \infty$, $s,\epsilon>0$ with $s+\epsilon<1$. Then for all $u\in B^{s+\epsilon}_{p,\infty}$ and $\xi\in \mathbb{R}^3\backslash \{0\}$ we have
    $$
        \lVert \delta u \rVert_{B^s_{p,\infty}} \lesssim |\xi|^{\epsilon} \lVert u \rVert_{B^{s+\epsilon}_{p,\infty}}.
    $$
\end{lemma}

\begin{proof}
First we observe that
$$
    \lVert \delta u \rVert_{L^p} = |\xi|^{\epsilon} \frac{1}{|\xi|^{\epsilon}} \lVert \delta u \rVert_{L^p} \leq |\xi|^{\epsilon} \sup_{h \in \mathbb{R}^3 \backslash \{0\}}\frac{1}{|h|^{\epsilon}} \lVert \delta_h u \rVert_{L^p} \leq |\xi|^{\epsilon} \lVert u \rVert_{B^{\epsilon}_{p,\infty}} \lesssim |\xi|^{\epsilon} \lVert u \rVert_{B^{s+\epsilon}_{p,\infty}}.
$$
Hence it remains to estimate the Besov seminorm. Due to the assumption $s+\epsilon<1$, this seminorm only contains a first-order increment. We obtain
\begingroup
\allowdisplaybreaks
\begin{align*}
    \lvert \delta u \rvert_{B^{s}_{p,\infty}} &= \sup_{h \in \mathbb{R}^3 \backslash \{0\}} \frac{1}{\lvert h \rvert^{s}}  \lVert \delta_\xi \delta_h u \rVert_{L^p} = \lvert \xi \rvert^{\epsilon} \sup_{h \in \mathbb{R}^3 \backslash \{0\}} \frac{1}{\lvert h \rvert^{s} \lvert \xi \rvert^{\epsilon}} \lVert \delta_\xi \delta_h u \rVert_{L^p} \\
    &\leq \lvert \xi \rvert^{\epsilon} \sup_{h \in \mathbb{R}^3 \backslash \{0\}, \lvert h \rvert \leq \lvert \xi \rvert} \frac{1}{\lvert h \rvert^{s} \lvert \xi \rvert^{\epsilon}} \lVert \delta_\xi \delta_h u \rVert_{L^p} + \lvert \xi \rvert^{\epsilon} \sup_{h \in \mathbb{R}^3 \backslash \{0\}, \lvert h \rvert \geq \lvert \xi \rvert} \frac{1}{\lvert h \rvert^{s} \lvert \xi \rvert^{\epsilon}} \lVert \delta_\xi \delta_h u \rVert_{L^p} \\
    &\leq \lvert \xi \rvert^{\epsilon} \sup_{h \in \mathbb{R}^3 \backslash \{0\}, \lvert h \rvert \leq \lvert \xi \rvert} \frac{1}{\lvert h \rvert^{s+\epsilon} }  \lVert \delta_\xi \delta_h u \rVert_{L^p} + \lvert \xi \rvert^{\alpha} \sup_{h \in \mathbb{R}^3 \backslash \{0\}, \lvert h \rvert \geq \lvert \xi \rvert} \frac{1}{ \lvert \xi \rvert^{s+\epsilon}} \lVert \delta_\xi \delta_h u \rVert_{L^p} \\
    &\leq 4 \lvert \xi \rvert^{\epsilon} \lVert  u \rVert_{B^{s+\epsilon}_{p,\infty}}.
\end{align*}
\endgroup
\end{proof}

\begin{proof}[Proof of Proposition~\ref{weakdefectzero}]
By Proposition \ref{weakzerodefect} we have to prove bound \eqref{weakzerodefecteq}. By assumption we have $2\alpha-1-s>0$, thus we can fix $\theta>0$ sufficiently small such that $2\alpha-1-s-3\theta>0$. Using Lemma~\ref{lemma:paradiff-summary} we obtain
\begingroup
\allowdisplaybreaks
\begin{align*}
    &\bigg\lvert \langle \delta \textbf{u} (\lvert \delta u \rvert^2 + \lvert \delta v \rvert^2 ), \psi \rangle_{B^{-s}_{1,\infty} \times B^s_{\infty,1} } \bigg\rvert \\
    &\leq \lVert \psi \rVert_{B^s_{\infty,1}} \lVert \delta \textbf{u} (\lvert \delta u \rvert^2 + \lvert \delta v \rvert^2 ) \rVert_{B^{-s}_{1,\infty}} \\
    &\leq \lVert \psi \rVert_{B^s_{\infty,1}} \lVert \delta \textbf{u} \rVert_{B^{-s}_{3,\infty}} \big\lVert |\delta u|^2 + |\delta v|^2 \big\rVert_{B^{s+\theta}_{3/2,\infty}} \\
    &\leq \lVert \psi \rVert_{B^s_{\infty,1}} \lVert \delta \textbf{u} \rVert_{B^{-s}_{3,\infty}} \Big(\lVert \delta u \rVert_{B^{\theta}_{3,\infty}} \lVert \delta u \rVert_{B^{s+2\theta}_{3,\infty}} + \lVert \delta v \rVert_{B^{\theta}_{3,\infty}} \lVert \delta v \rVert_{B^{s+2\theta}_{3,\infty}} \Big).
\end{align*}
\endgroup
Now in order to apply Lemma~\ref{lemma:besov-delta} we observe that we may assume that $s\leq \frac{1}{2}$, see Remark~\ref{paraproductremark3}, and that $\alpha<1$. So we deduce from Lemma~\ref{lemma:besov-delta} that for all $0 \neq \xi \in \mathbb{R}^3$
\begin{equation*}
    \lVert \delta u \rVert_{B^{\theta}_{3,\infty}} \leq \lvert \xi \rvert^{\alpha - \theta} \lVert u \rVert_{B^{\alpha}_{3,\infty}}, \quad \lVert \delta u \rVert_{B^{s + 2 \theta}_{3,\infty}} \leq \lvert \xi \rvert^{\alpha - 2\theta - s} \lVert u \rVert_{B^{\alpha}_{3,\infty}},
\end{equation*}
and similarly for $\delta v$. We therefore conclude that
\begin{equation*}
    \bigg\lvert \langle \delta \textbf{u} (\lvert \delta u \rvert^2 + \lvert \delta v \rvert^2 ), \psi \rangle_{B^{-s}_{1,\infty} \times B^s_{\infty,1} } \bigg\rvert \leq \lvert \xi \rvert^{2 \alpha-s-3\theta} \lVert \psi \rVert_{B^s_{\infty,1}} \lVert \delta \textbf{u} \rVert_{B^{-s}_{3,\infty}}  (\lVert  u \rVert_{B^{\alpha}_{3,\infty}}^2 + \lVert  v \rVert_{B^{\alpha}_{3,\infty}}^2).
\end{equation*}
By assumption it holds that $\lVert \delta \textbf{u} \rVert_{B^{-s}_{3,\infty}}  (\lVert  u \rVert_{B^{\alpha}_{3,\infty}}^2 + \lVert  v \rVert_{B^{\alpha}_{3,\infty}}^2) \in L^1 (0,T)$, moreover we can take $\sigma (\lvert \xi \rvert) \coloneqq \lvert \xi \rvert^{2\alpha -1 -s-3\theta}$, which satisfies the assumptions of Proposition \ref{weakzerodefect}. Therefore by Proposition \ref{weakzerodefect} we conclude that $D (\textbf{u}) = 0$.
\end{proof}

Finally, we are able to prove conservation of energy under sufficient regularity analogously to Proposition \ref{exponentlemma}.
\begin{theorem} \label{veryweakconservation}
Let $\mathbf{u}$ be a type III weak solution of the hydrostatic Euler equations with regularity parameter $s$, and that $w\in L^3((0,T);B^{-s}_{3,\infty}(\mathbb{T}^3))$ and $u, v \in L^{3} ((0,T); B^\alpha_{3, \infty} (\mathbb{T}^3))$ with $\alpha > \frac{1}{2} + \frac{s}{2}$. Then there is conservation of energy. In other words, it holds that
\begin{equation} \label{conservationequation}
\lVert u(t_1, \cdot ) \rVert_{L^2}^2 + \lVert v(t_1, \cdot) \rVert^2_{L^2} = \lVert u(t_2, \cdot ) \rVert_{L^2}^2 + \lVert v(t_2, \cdot) \rVert^2_{L^2},
\end{equation}
for almost all $t_1, t_2 \in (0,T)$.
\end{theorem}
\begin{proof}
By Proposition \ref{weakdefectzero} we know that $D (\textbf{u}) = 0$ under the regularity assumptions. Therefore we have the following equation of local energy balance
\begin{align*}
&\int_0^T \bigg[ \int_{\mathbb{T}^3} \bigg( u^2 \partial_t \psi + v^2 \partial_t \psi  \bigg) dx + \bigg\langle  (u^2 + v^2 + 2p) \textbf{u}, \nabla \psi \bigg\rangle_{B^{-s}_{1,\infty} \times B^s_{\infty,1}}  \bigg] dt = 0.
\end{align*}
We conclude proceeding exactly as in the proof of Proposition \ref{exponentlemma}.
\end{proof}
We observe that the there is an upper bound on the Onsager exponent, namely $\frac{2}{3}$. We will show this in the following proposition.
\begin{proposition} \label{upperbound}
Let $\mathbf{u}$ be a type III weak solution of the hydrostatic Euler equations with (arbitrary) regularity parameter $s$ such that $u,v \in L^3 ((0,T); B^{\alpha}_{3,\infty} (\mathbb{T}^3))$ with $\alpha > \frac{2}{3}$. Then conservation of energy (as in equation \eqref{conservationequation}) holds.
\end{proposition}

\begin{proof}
We prove that the sufficient condition for energy conservation of Proposition \ref{weakzerodefect} is satisfied. By equation \eqref{equationw} and Lemma \ref{primitivelemma} we know that $w \in L^3 ((0,T); B^{\alpha - 1}_{3,\infty} (\mathbb{T}^3))$. As $\alpha-1>-\frac{1}{3}$, this implies that $w \in L^3 ((0,T); B^{-1/3}_{3,\infty} (\mathbb{T}^3))$ and $\mathbf{u}$ is a type III weak solution with regularity parameter $s=1/3$. Now because $\frac{1}{2}+\frac{s}{2}=\frac{2}{3}<\alpha$, the claim follows immediately from Theorem~\ref{veryweakconservation}.
\end{proof}

One can infer from this result that the notion of a type III weak solution will only have implications on the Onsager conjecture in the range $0 < s < \frac{1}{3}$, because if $s \geq \frac{1}{3}$ Theorem \ref{veryweakconservation} follows from Proposition \ref{upperbound}.

\subsection{Type II weak solutions}
In this section, we introduce type II weak solutions and state sufficient conditions for such solutions to conserve energy.
\begin{definition} \label{typeIIdefinition}
A type II weak solution consists of a velocity field $\mathbf{u} = (u,v,w) : \mathbb{T}^3 \times (0,T) \rightarrow \mathbb{R}^3$ as well as a pressure $p : \mathbb{T}^3 \times (0,T) \rightarrow \mathbb{R}$ together with a regularity parameter $0 < s < \frac{1}{2}$, if the following conditions hold:
\begin{itemize}
    \item $w \in L^2 ((0,T); L^2 (\mathbb{T}; B^{-s}_{2,\infty} (\mathbb{T}^2)))$, $u, v \in L^\infty ((0,T); L^2 (\mathbb{T}^3)) \cap L^2 ((0,T); L^2 (\mathbb{T}; B^{\sigma'}_{2,\infty} (\mathbb{T}^2)))$ for some $\sigma' > s$ and where $L^2 (\mathbb{T})$ refers to the regularity in the vertical direction.
    \item For all $\phi_1, \phi_2 \in \mathcal{D} (\mathbb{T}^3 \times (0,T); \mathbb{R})$ it holds that
    \begingroup
    \allowdisplaybreaks
    \begin{align}
    &\int_0^T \int_{\mathbb{T}^3} u \partial_t \phi_1 d x \; dt + \int_0^T \int_{\mathbb{T}} \langle u \textbf{u} , \nabla \phi_1 \rangle dz \; dt + \int_0^T \int_{\mathbb{T}^3} \Omega v \phi_1 d x \; dt  \\
    &+ \int_0^T \int_{\mathbb{T}^3} p \partial_x \phi_1 d x \; dt = 0, \nonumber \\
    &\int_0^T \int_{\mathbb{T}^3} v \partial_t \phi_2 d x \; dt + \int_0^T \int_{\mathbb{T}} \langle v \textbf{u} , \nabla \phi_2 \rangle dz \; dt - \int_0^T \int_{\mathbb{T}^3} \Omega u \phi_2 d x \; dt \\
    &+ \int_0^T \int_{\mathbb{T}^3} p \partial_y \phi_2 d x \; dt = 0. \nonumber
    \end{align}
    \endgroup
    The brackets $\langle \cdot, \cdot \rangle$ refer to the distributional duality between $\mathcal{D}' (\mathbb{T}^2)$ and $\mathcal{D} (\mathbb{T}^2)$.
    \item For all $\phi_3 \in \mathcal{D} (\mathbb{T}^3 \times (0,T))$ it holds that
    \begin{equation} 
    \int_0^T \int_{\mathbb{T}^3} p \partial_z \phi_3 dx dt = 0.
    \end{equation}
    \item The velocity field $\mathbf{u}$ is divergence-free, in particular for all $\phi_4 \in \mathcal{D} (\mathbb{T}^3 \times (0,T))$ it holds that
    \begin{equation}
    \int_0^T \int_{\mathbb{T}} \langle \textbf{u} , \nabla \phi_4 \rangle dz dt = 0.
    \end{equation}
    \item Finally, for almost all $t \in (0,T)$ it holds that
    \begin{equation} \label{typeIIboundary}
    w(\cdot, 0,t) = w(\cdot, 1, t) = 0 \quad \text{in } B^{-s}_{2,\infty} (\mathbb{T}^2).
    \end{equation}
\end{itemize}
\end{definition}
In a similar fashion as it was argued in Remark \ref{paraproductremark}, one can prove that $u w (\cdot, z,t), v w (\cdot, z, t) \linebreak \in B^{-s}_{1,\infty} (\mathbb{T}^2)$. We skip the details here. One should observe that the vertical velocity is more regular for type II weak solutions compared to type III weak solutions, whereas the opposite holds for the horizontal velocities. Hence, neither notion is stronger than the other if $0 < s < \frac{1}{2}$. In the case $s > \frac{1}{2}$, type III weak solutions are also type II weak solutions. Finally, we observe that boundary condition \eqref{typeIIboundary} can be made sense of in a similar fashion as it was argued in Remark \ref{bcremark}.

Next we consider sufficient conditions for energy conservation. We only state the results and omit the proofs, as they are very similar in nature as for the type III weak solutions, see Lemma \ref{genweaktestfunctions}, Theorem \ref{weakenergyequation}, Propositions \ref{weakzerodefect} and \ref{weakdefectzero}, Theorem \ref{veryweakconservation} and Proposition \ref{upperbound}. We first state the equation of local energy balance.
\begin{theorem}
Let $\mathbf{u}$ be a type II weak solution of the hydrostatic Euler equations with regularity parameter $s$, such that $w \in L^3 ((0,T); L^3 (\mathbb{T}; B^{-s}_{3,\infty} (\mathbb{T}^2)))$ and $u,v \in L^3 ((0,T); L^3 (\mathbb{T}; B^{\sigma'}_{3,\infty} (\mathbb{T}^3)))$ for some $\sigma' > s$. Then for all $\psi \in \mathcal{D} (\mathbb{T}^3 \times (0,T))$ the following equation of local energy balance is satisfied
\begin{align*}
&\int_0^T \bigg[ \int_{\mathbb{T}^3} \bigg( u^2 \partial_t \psi + v^2 \partial_t \psi  \bigg) dx + \int_{\mathbb{T}} \bigg\langle  (u^2 + v^2 + 2p) \mathbf{u}, \nabla \psi \bigg\rangle_{B^{-s}_{1,\infty} \times B^s_{\infty,1}} dz\nonumber \\
&- \frac{1}{2} \int_{\mathbb{T}} \bigg\langle  D(\mathbf{u}), \psi \bigg\rangle_{B^{-s}_{1,\infty} \times B^s_{\infty,1} } dz  \bigg] dt = 0.
\end{align*}
The defect term $D (\mathbf{u})$ is given by
\begin{equation*}
D (\mathbf{u}) \coloneqq \lim_{\epsilon \rightarrow 0} \int_{\mathbb{R}^3} d \xi \bigg[ \nabla \varphi_\epsilon (\xi) \cdot \delta \mathbf{u} (\xi ; x, t) (\lvert\delta u (\xi ; x, t) \rvert^2 + \lvert\delta v (\xi ; x, t) \rvert^2) \bigg],
\end{equation*}
and is independent of the choice of mollifier.
\end{theorem}
\begin{remark}
Note that the duality brackets are now on the two-dimensional torus $\mathbb{T}^2$, while for the type III weak solutions they were three-dimensional.
\end{remark}
The usual sufficient condition for the defect term to be zero is given below.
\begin{lemma}
Let $\mathbf{u}$ be a type II weak solution of the hydrostatic Euler equations with regularity parameter $s$, such that $w \in L^3 ((0,T); L^3 (\mathbb{T}; B^{-s}_{3,\infty} (\mathbb{T}^2)))$ and $u,v \in L^3 ((0,T); L^3 (\mathbb{T}; B^{\sigma'}_{3,\infty} (\mathbb{T}^3)))$ for some $\sigma' > s$. Moreover, assume that the following inequality holds for all $\psi \in \mathcal{D} (\mathbb{T}^3 \times (0,T))$
\begin{equation*}
\bigg\lvert \int_0^1 \langle \delta \mathbf{u} (\lvert \delta u \rvert^2 + \lvert \delta v \rvert^2 ), \psi \rangle_{B^{-s}_{1,\infty} \times B^s_{\infty,1} } dz \bigg\rvert \leq C(t) \lvert \xi \rvert \sigma (\lvert \xi \rvert) \int_0^1 \lVert \psi \rVert_{B^s_{\infty,1} (\mathbb{T}^2)} dz,
\end{equation*}
where $C \in L^1 (0,T)$ and $\sigma \in L^\infty_{\mathrm{loc}} (\mathbb{R})$ has the property that $\sigma (\lvert \xi \rvert) \rightarrow 0$ as $\lvert \xi \rvert \rightarrow 0$. Then $D (\mathbf{u}) = 0$, which in turn implies conservation of energy.
\end{lemma}
Finally, we can state the sufficient condition for energy conservation for type II weak solutions.
\begin{proposition} \label{typeIIconservation}
Let $\mathbf{u}$ be a type II weak solution of the hydrostatic Euler equations with regularity parameter $s$ such that $w \in L^3 ((0,T); L^3 (\mathbb{T}; B^{-s}_{3,\infty} (\mathbb{T}^2)))$ and $u,v \in L^3 ((0,T); B^{\alpha}_{3,\infty} (\mathbb{T}; \linebreak B^\beta_{3,\infty} (\mathbb{T}^2)))$ with $\alpha > \frac{1}{2}$ and $\beta > \frac{1}{2} + \frac{s}{2}$, then the solution conserves energy. In particular, equation \eqref{conservationequation} holds.
\end{proposition}
Like in the case of type III weak solution, we have a result which is analogous to Proposition \ref{upperbound}.
\begin{proposition} \label{typeIIupperbound}
Let $\mathbf{u}$ be a type II weak solution of the hydrostatic Euler equations with the property that $u,v \in L^3 ((0,T); B^{\alpha}_{3,\infty} (\mathbb{T}; B^\beta_{3,\infty} (\mathbb{T}^2)))$ with $\alpha > \frac{1}{2}$ and $\beta > \frac{2}{3}$ (note that this space can be defined similarly to equation \eqref{anisotropicbesovnorm}). Then the weak solution conserves energy.
\end{proposition}
\section{Sufficient conditions in terms of Sobolev spaces} \label{sobolevappendix}
In the next two sections, we give sufficient conditions for type I weak solutions (or at least with similar regularity requirements to type I) to conserve energy solely in terms of the horizontal velocities. In the previous sections, we started with a regularity assumption on the vertical velocity and then derived sufficient conditions for energy conservation in terms of the horizontal velocities (with sometimes additional assumptions on the vertical velocity). However, these assumptions are not independent of each other by the anisotropic regularity and the nonlocality imposed by \eqref{equationw}.

The fact that the assumption that $w \in L^2 (\mathbb{T}^3 \times (0,T))$ induces implicit assumptions on the horizontal velocities can be seen as follows. The fact that the vertical velocity $w$ has $L^2$ spatial regularity means that either $\partial_x u$ and $\partial_y v$ are in $L^2 ( (0,T); L^2 (\mathbb{T}^3))$ or their potential lack of regularity cancels in a particular way that is unclear at this point. In order to gain a better understanding of this issue we want to look at the required regularity assumption on the horizontal velocities $u$ and $v$ such that a regularity assumption on $w$ is no longer necessary for defining a weak solution or proving conservation of energy.

In this section we will prove such sufficient conditions in terms of Sobolev spaces and Lebesgue spaces. These results complement our earlier results. As was shown in Propositions \ref{zerodefectlemma} and \ref{exponentlemma}, condition \eqref{zerodefect} is sufficient for a type I weak solution with $w \in L^2 (\mathbb{T}^3 \times (0,T))$ to conserve energy. However, if we would like to avoid any explicit assumptions on $w$, two separate conditions on $u$ and $v$ are necessary:
\begin{enumerate}
    \item The horizontal velocities $u$ and $v$ need to have sufficient H\"older regularity with a given exponent such that the function $\sigma$ in equation \eqref{zerodefect} is $o(1)$. This automatically implies that their H\"older exponent must be bigger than $\frac{1}{2}$.
    \item The vertical velocity $w$ needs to satisfy some integrability condition in order for the inequality \eqref{zerodefect} to hold, because $\delta \mathbf{u}$ needs to be bounded in some $L^p$ space. Note that we want to proceed here as we did for type I weak solutions.
\end{enumerate}
From this reasoning we can conclude that we can either assume that $u,v$ are in a sufficiently regular Sobolev space such that both conditions are satisfied (by use of the Sobolev embedding theorem), or we have to make two separate assumptions.
\begin{proposition} \label{Sobolevexponent}
Let $u , v \in L^3 ( (0,T); C^{0, \beta} (\mathbb{T}^3) \cap W^{1,p} (\mathbb{T}^3))$ with $\beta > \frac{1}{2}$ and $p>1$ or alternatively $u , v  \in L^3 ((0,T); W^{1,p} (\mathbb{T}^3))$ with $p > 6$. Moreover, assume that $\int_{\mathbb{T}} \nabla_H \cdot \mathbf{u}_H dz = 0$ and that together with the associated $w$ (via equation \eqref{equationw}) the velocity field $\mathbf{u}$ satisfies the weak formulation in equations \eqref{weaku}-\eqref{pressureweak}. Then $D(\mathbf{u}) = 0$. In particular, the weak solution conserves energy (see equation \eqref{energyconservationeq}).
\end{proposition}
\begin{proof}
We first observe that the second case reduces to the first case by the Sobolev embedding theorem, therefore we will only consider the first case. We first show that the defect term is bounded. Now since $u(\cdot, t), v (\cdot, t) \in W^{1,p} (\mathbb{T}^3)$, we know that $\partial_x u (\cdot, t) + \partial_y v (\cdot, t) \in L^p (\mathbb{T}^3)$ and by virtue of equation \eqref{equationw} and Lemma \ref{primitivelemma} it follows that $w \in L^3 ( (0,T); L^p (\mathbb{T}^3))$ with $p > 1$. Using H\"older's inequality we obtain that
\begin{align*}
&\lVert \delta \textbf{u} (\xi; \cdot, t) ( \lvert \delta u  (\xi; \cdot, t) \rvert^2 + \lvert\delta v (\xi; \cdot, t)  \rvert^2 ) \rVert_{L^1} \leq \lVert \delta \textbf{u} (\xi; \cdot, t) \rVert_{L^p} \lVert \lvert\delta u  (\xi; \cdot, t) \rvert^2 + \lvert\delta v (\xi; \cdot, t) \rvert^2 \rVert_{L^{p/(p-1)} } \\
&\leq \lVert \delta \textbf{u} (\xi; \cdot, t) \rVert_{L^p} \bigg( \lVert \delta u  (\xi; \cdot, t) \rVert_{L^{2p/(p-1)} }^2 + \lVert \delta v (\xi; \cdot, t) \rVert_{L^{2p/(p-1)} }^2 \bigg).
\end{align*}
Observe that $u(\cdot, t) , v(\cdot, t) \in L^{2p/(p-1)} (\mathbb{T}^3)$ because of the fact that $u(\cdot, t), v(\cdot, t) \in C^{0, \beta} (\mathbb{T}^3)$ with $\beta > \frac{1}{2}$.

In addition we see that
\begin{equation*}
\lVert \delta \textbf{u} (\cdot, t) \rVert_{L^p} \bigg( \lVert \delta u  (\cdot, t) \rVert_{L^{2p/(p-1)} }^2 + \lVert \delta v (\cdot, t) \rVert_{L^{2p/(p-1)} }^2 \bigg) \in L^1 (0,T),
\end{equation*}
because we assumed that $u,v$ and $w$ have $L^3$ temporal regularity.

Finally we need to check condition \eqref{zerodefect}. We see that
\begin{align*}
&\lVert \delta \textbf{u} (\xi; \cdot, t) ( \lvert\delta u (\xi; \cdot, t) \rvert^2 + \lvert\delta v (\xi; \cdot, t) \rvert^2 ) \rVert_{L^1} \lesssim \lvert \xi \rvert^{2 \beta} \lVert \delta \textbf{u} (\xi; \cdot, t) \rVert_{L^p} \\
&\cdot \bigg( \bigg\lVert \frac{\delta u }{\lvert \xi \rvert^{ \beta}} (\xi; \cdot, t) \bigg\rVert_{L^{2p/(p-1)} }^2 + \bigg\lVert \frac{\delta v}{\lvert \xi \rvert^{ \beta}} (\xi; \cdot, t) \bigg\rVert_{L^{2p/(p-1)} }^2 \bigg) \\
&\lesssim 4 \lvert \xi \rvert^{ 2 \beta} \lVert \textbf{u} (\cdot, t) \rVert_{L^p} \lVert 1 \rVert_{L^{2p/(p-1)}}^2 \big(\lVert u (\cdot, t) \rVert_{C^{0, \beta}}^2 + \lVert v (\cdot, t) \rVert_{C^{0, \beta}}^2 \big).
\end{align*}
Thus by Proposition \ref{zerodefectlemma} $D (\textbf{u}) = 0$. Now energy conservation can be proven analogously to the proof of Proposition \ref{exponentlemma}.
\end{proof}
\section{Sufficient conditions in terms of Besov spaces} \label{removalappendix}
As was already mentioned in the previous section, the regularity assumption on $w$ contains (by means of equation \eqref{equationw}) implicit assumptions on the horizontal velocities $u$ and $v$. In this section we prove a criterion in terms of Besov spaces which is similar to the one in Proposition \ref{Sobolevexponent}.
\begin{proposition} \label{besovtheorem}
Assume that $u, v \in L^3 ((0,T);  B^{s}_{\frac{9}{4}, \infty} (\mathbb{T}^3))$ with $s > 1$ and $\int_{\mathbb{T}} \nabla_H \cdot \mathbf{u}_H dz = 0$, such that they (together with the associated $w$ via equation \eqref{equationw}) satisfy the weak formulation given in equations \eqref{weaku}-\eqref{pressureweak}. Then $D (\mathbf{u}) = 0$, which implies that energy is conserved (see equation \eqref{energyconservationeq}).
\end{proposition}
\begin{proof}
We first recall the following property of Besov spaces \cite{leoni,bahouri}
\begin{equation} \label{besovderivative}
\lVert \partial^\alpha f \rVert_{B^{s - \lvert \alpha \rvert }_{p,q}} \leq \lVert f \rVert_{B^s_{p,q}}, \quad \lvert \alpha \rvert < s.
\end{equation}
This implies that $\partial_z w$ and hence $w \in L^3 ( (0,T); B^{s-1}_{\frac{9}{4}, \infty} (\mathbb{T}^3))$ by equation \eqref{equationw} and by applying Lemma \ref{primitivelemma}, in particular we obtain that $w \in L^3 ((0,T); L^{9/4} (\mathbb{T}^3))$. Subsequently by a Besov embedding \cite{leoni} we can conclude that
\begin{equation*}
B^s_{9/4, \infty} (\mathbb{T}^3) \subset B^{s-1/2}_{18/5, \infty} (\mathbb{T}^3), \quad  s>1.
\end{equation*}
Therefore $u,v \in L^3 ((0,T); B^{r}_{18/5, \infty} (\mathbb{T}^3))$ with $r > \frac{1}{2}$.
Now we estimate the defect term given in equation \eqref{defectdefinition} and use Proposition \ref{zerodefectlemma} to prove it is zero. By using H\"older's inequality we can derive that
\begin{align*}
&\int_{\mathbb{T}^3} \lvert \delta \textbf{u} (\xi ; x,t) \rvert \bigg(\lvert  \delta u (\xi ; x,t) \rvert^2 + \lvert  \delta v(\xi ; x,t) \rvert^2 \bigg) dx \\
&\leq \lVert \delta \textbf{u} (\xi ; \cdot,t) \rVert_{L^{9/4}} \big( \lVert \delta u (\xi ; \cdot,t) \rVert_{L^{18/5}}^2 + \lVert \delta v (\xi ; \cdot,t) \rVert_{L^{18/5}}^2 \big).
\end{align*}
We then get
\begin{align*}
&\lVert \delta \textbf{u} (\xi ; \cdot,t) \rVert_{L^{9/4}} ( \lVert \delta u (\xi ; \cdot,t) \rVert_{L^{18/5}}^2 + \lVert \delta v (\xi ; \cdot,t) \rVert_{L^{18/5}}^2 ) \\
&\leq 2 \lVert \textbf{u} (\cdot,t) \rVert_{L^{9/4}} \lvert \xi \rvert^{2 r} \bigg( \bigg\lVert \frac{\delta u (\xi ; \cdot,t)}{\lvert \xi \rvert^r} \bigg\rVert_{L^{18/5}}^2 + \bigg\lVert \frac{\delta v (\xi ; \cdot,t)}{\lvert \xi \rvert^r} \bigg\rVert_{L^{18/5}}^2 \bigg) \\
&\leq 2 \lVert \textbf{u} (\cdot, t) \rVert_{B^{s-1}_{9/4, \infty}} \lvert \xi \rvert^{2 r} \bigg( \lVert  u (\cdot, t) \rVert_{B^r_{18/5, \infty}}^2 + \lVert  v (\cdot, t) \rVert_{B^r_{18/5, \infty}}^2 \bigg).
\end{align*}
Hence condition \eqref{zerodefect} is satisfied. Finally, we need to prove the regularity in time. From the derived conclusions that $u,v \in L^3 ((0,T); B^r_{18/5, \infty} (\mathbb{T}^3))$ (with $r > \frac{1}{2}$) and $w \in L^3 ((0,T); \linebreak B^{s-1}_{9/4,\infty} (\mathbb{T}^3))$ (and $s > 1$) we observe that
\begin{equation*}
\lVert \textbf{u} (\cdot, t) \rVert_{B^{s-1}_{9/4, \infty}} \bigg( \lVert  u (\cdot, t) \rVert_{B^r_{18/5, \infty}}^2 + \lVert  v (\cdot, t) \rVert_{B^r_{18/5, \infty}}^2 \bigg) \in L^1  (0,T) .
\end{equation*}
Thus we conclude by Proposition \ref{zerodefectlemma} that $D (\textbf{u}) = 0$. Then by a similar argument to Proposition \ref{exponentlemma} we find that the energy is conserved.
\end{proof}
\begin{remark}
From the Besov embedding theorem \cite{leoni} we can derive that
\begin{equation*}
B^s_{9/4,\infty} (\mathbb{T}^3) \subset B^t_{4, \infty} (\mathbb{T}^3)
\end{equation*}
where $s > 1$ and $t > \frac{7}{12}$. It then obviously follows that the same embedding is true if $t > \frac{1}{2}$. Therefore the result derived in Proposition \ref{exponentlemma} is stronger than the results from this section, but for Proposition \ref{exponentlemma} we needed a separate regularity assumption on $w$.
\end{remark}
\begin{remark}
Proposition \ref{besovtheorem} is in the context of type I weak solutions. It is also possible to consider sufficient regularity conditions to both define a weak solution  and prove energy conservation in the context of type II and III weak solutions. This was discussed to some extent in Remark \ref{paraproductremark3} and Proposition \ref{upperbound}.
\end{remark}

\section{Conclusion} \label{conclusion}
In this paper, we have proven that the Onsager exponent of the hydrostatic Euler equations is at most $\frac{1}{2}$ (for type I weak solutions). This result differs from the original Onsager conjecture for the Euler equations, where it has been established that the threshold for energy conservation is $\frac{1}{3}$. As it was mentioned before, a technical explanation for this increase can be given by looking again at the defect term
\begin{equation*}
D (\textbf{u}) (x,t) \coloneqq \frac{1}{2} \lim_{\epsilon \rightarrow 0} \int_{\mathbb{R}^3} \bigg[ \nabla \varphi_\epsilon (\xi) \cdot \delta \textbf{u} (\xi;x,t) (\lvert \delta u  (\xi;x,t) \rvert^2 + \lvert \delta v  (\xi;x,t) \rvert^2) \bigg] d \xi .
\end{equation*}
It is clear from Proposition \ref{zerodefectlemma} that for the expression to be zero, the term $\sigma$ in equation \eqref{zerodefect} has to be $o(1)$. The convergence to zero of $\sigma$ as $\lvert \xi \rvert \rightarrow 0$ can be very slight however. This was illustrated by the consideration of the logarithmic H\"older space in Definition \ref{logholderdef} and the proof that the defect term is zero in Proposition \ref{loglemma}.

The difference with the defect term for the Euler equations, as stated in \cite{duchon}, is that in the latter case there is a cubic term $\lvert \delta \textbf{u} \rvert^3$ and that the vertical velocity has the same regularity as the horizontal velocities in that case. This means that for the Euler equations the Besov regularity necessary for $\sigma$ to be $o(1)$ can be equally distributed among the three terms in the product $\lvert \delta \textbf{u} \rvert^3$, leading to an Onsager exponent of $\frac{1}{3}$.

For the hydrostatic Euler equations, due to the anistropic nature of the equations, the required Besov regularity (for $\sigma$ to be $o(1)$ in equation \eqref{zerodefect}) must be distributed on the term $\lvert\delta \mathbf{u}_H \rvert^2 $, which leads to an Onsager exponent of $\frac{1}{2}$.

We would also like to summarise the different results presented in this paper. Unlike the Euler equations, the hydrostatic Euler equations seem to have a `family' of Onsager conjectures. We now give an overview of these:
\begin{enumerate}
    \item We first consider type I weak solutions (cf. Definition \ref{weaksolutiondefinition}), i.e. we assume that $w \in L^2 ((0,T); L^2 (\mathbb{T}^3))$. Then there are several subcases:
    \begin{itemize}
        \item Without any further assumptions, the Onsager exponent is $\frac{1}{2}$, cf. Proposition \ref{exponentlemma}.
        \item Another sufficient condition for energy conservation is that $u,v \in L^4 ((0,T); \linebreak C^{1/2}_{\mathrm{log}} (\mathbb{T}^3))$, cf. Proposition \ref{loglemma}.
        \item If we assume that $w$ is Besov regular with exponent $\beta$, then the Onsager exponent for the horizontal velocities is $\frac{1}{2} - \frac{1}{2} \beta$, cf. Proposition \ref{holderlemma}.
        \item If we assume that the horizontal velocities $u$ and $v$ are Besov regular with exponent $\alpha$ in the $z$-direction (where $\frac{1}{3} < \alpha < \frac{1}{2}$) and we make no additional assumption on $w$, then the Onsager exponent in the horizontal directions is $2 - 2 \alpha$, cf. Proposition \ref{directionlemma}.
    \end{itemize}
    \item Under the assumption that $w \in L^2 ((0,T); B^{-s}_{2,\infty} (\mathbb{T}^3))$ (for $s \leq \frac{1}{2}$), we can make sense of the equation (with additional regularity assumptions on $u$ and $v$). This is in the context of type III weak solutions, as introduced in Definition \ref{veryweakdef}. In the range $0 < s \leq \frac{1}{3}$, the Onsager exponent is $\frac{1}{2} + \frac{s}{2}$ (cf. Proposition \ref{weakdefectzero}), where the regularity exponents refer to $L^3$-based Besov spaces. For $s \geq \frac{1}{3}$, the Onsager exponent is $\frac{2}{3}$ (cf. Proposition \ref{upperbound}).
    \item In section \ref{veryweakappendix} we also considered the type II weak solutions, for which $w \in L^2 ((0,T); L^2 (\mathbb{T}; \linebreak B^{-s}_{2,\infty} (\mathbb{T}^2)))$, see Definition \ref{typeIIdefinition}. For such solutions we proved that they conserve energy if $w \in L^3 ((0,T); L^3 (\mathbb{T}; B^{-s}_{3,\infty} (\mathbb{T}^3)))$ and $u,v \in L^3 ((0,T); B^{\alpha}_{3,\infty} (\mathbb{T}; B^\beta_{3,\infty} (\mathbb{T}^2)))$ with $\alpha > \frac{1}{2}$ and $\beta > \frac{1}{2} + \frac{s}{2}$, cf. Proposition \ref{typeIIconservation}. Moreover, we proved that if $\alpha > \frac{1}{2}$ and $\beta > \frac{2}{3}$ the solution conserves energy, regardless of assumptions on $w$ (cf. Proposition \ref{typeIIupperbound}).
    \item Finally, there is the case where we do not make explicit regularity assumptions on $w$ (neither for making sense of the equation nor for proving energy conservation) but only on the horizontal velocities $u$ and $v$. Then we considered two subcases in sections \ref{sobolevappendix} and \ref{removalappendix}:
    \begin{itemize}
        \item In terms of Sobolev and H\"older spaces, we proved that $u, v \in L^3 ((0,T); C^{0, \beta} (\mathbb{T}^3) \cap W^{1,p} (\mathbb{T}^3))$ with $p > 1$ and $\beta > \frac{1}{2}$, or alternatively $u,v \in L^3 ((0,T); W^{1,p} (\mathbb{T}^3)) $ with $p > 6$ suffices to define a $w$ such that $(u,v,w)$ is a type I weak solution which conserves energy (cf. Proposition \ref{Sobolevexponent}).
        \item Moreover, we proved that if $u,v \in L^3 ((0,T); B^s_{9/4,\infty} (\mathbb{T}^3))$ with $s > 1$ then we can define $w$ such that $(u,v,w)$ is a type I weak solution which conserves energy (cf. Proposition \ref{besovtheorem}).
    \end{itemize}
    Note that the result of Proposition \ref{upperbound} is related to the results of sections \ref{sobolevappendix} and \ref{removalappendix}. Namely, that for $u,v \in L^3 ((0,T); B^\alpha_{3,\infty} (\mathbb{T}^3))$ with $\alpha > \frac{2}{3}$ it is implied that $w \in L^3 ((0,T); B^{-1/3}_{3,\infty} (\mathbb{T}^3))$ such that $(u,v,w)$ is a type III weak solution with regularity parameter $\frac{1}{3}$ which conserves energy, under the condition that the associated $w$ is periodic and satisfies the symmetry conditions.
\end{enumerate}

The reason we have considered these different cases is that the vertical velocity $w$ no longer has its own dynamical equation but is constrained by the incompressibility instead, therefore the equations are anisotropic and nonlocal. In order to establish an equation of local energy balance, a regularity assumption on $w$ is necessary. One needs the vertical velocity to have sufficient regularity to make sense of a weak solution but on the other hand regularity assumptions on $w$ impose implicit regularity assumptions on $u$ and $v$. This leads us to consider different types of weak solutions to the hydrostatic Euler equations.

\begin{remark} \label{turbulenceremark}
In \cite[~Theorem 2]{drivas} a theorem for the incompressible Navier-Stokes equations was established that describes the absence of anomalous dissipation in the vanishing viscosity limit under sufficient regularity assumptions. If one would assume sufficient regularity for the horizontal velocities uniform in viscosity to rule out anomalous dissipation in the viscous primitive equations, we expect the limiting solution of the inviscid primitive equations to be a type III weak solution rather than a type I weak solution. \\
This is due to the lack of an a priori bound of the $L^2 (\mathbb{T}^3)$ norm of the vertical velocity $w$. In particular, if one assumes solutions of the viscous primitive equations to satisfy uniform regularity bounds in viscosity, say $u, v \in L^3 ((0,T); B^{\alpha}_{3,\infty} (\mathbb{T}^3))$ for $\frac{1}{2} < \alpha < 1$, then by the incompressibility condition the vertical velocity $w$ only has uniform regularity in $L^3 ((0,T); B^{\alpha-1}_{3,\infty} (\mathbb{T}^3))$. Therefore we expect that the type III notion of weak solution has more implications on turbulence than the type I notion. \\
As was noted in the introduction, for classical solutions of the inviscid primitive equations the physical boundary-value problem in the channel with no-normal flow boundary conditions is equivalent to the problem posed on the unit torus with the symmetry assumptions
\begin{equation}
w \text{ odd in } z, \quad u, v \text{ and } p \text{ even in } z.
\end{equation}
In the case of the viscous primitive equations, for strong solutions the problem posed on the torus with the aforementioned symmetry assumptions is equivalent to the physical boundary-value problem in the channel with no-permeability and stress-free boundary conditions (see \cite{caohorizontal})
\begin{equation}
w (x,y,z,t) \lvert_{z = 0, L} = \partial_z u (x,y,z,t) \lvert_{z = 0, L} = \partial_z v (x,y,z,t) \lvert_{z = 0, L} = 0.
\end{equation}
For this reason, even if one wishes to formulate sufficient conditions for the absence of anomalous dissipation in the viscous primitive equations, it is possible to consider the problem on the torus with symmetry conditions instead of in the channel with physical boundaries (as in this case the boundary is `virtual' in some sense). Finally, we note that in the case of Dirichlet boundary conditions, studying sufficient conditions to rule out anomalous dissipation for the viscous primitive equations becomes more subtle due to the appearance of boundary layers.
\end{remark}

There are mainly two ways to establish sufficient conditions for energy conservation. One either starts with a regularity assumption on $w$ to properly define a weak solution and then derive sufficient conditions for $u$ and $v$ to ensure conservation of energy. Alternatively, one can impose sufficient conditions on $u$ and $v$ to define a weak solution and then obtain sufficient conditions for energy conservation. This range of possibilities leads to a `family' of Onsager conjectures, which was summarised above and is the content of this contribution.

\begin{remark}
We observe that it is possible to even extend this `family' of Onsager conjectures by combining the different ideas we used in this paper. For example, it is possible to derive the horizontal and vertical Onsager exponents for type III weak solutions, similar to what was done in Proposition \ref{directionlemma} for type I weak solutions of the hydrostatic Euler equations.
\end{remark}

\section*{Acknowledgements}
The authors would like to thank the anonymous referee for the thorough reading of the paper as well as the useful comments and suggestions. D.W.B. acknowledges support from the Cambridge Trust, the Cantab Capital Institute for Mathematics of Information and the Hendrik Muller fund. S.M. acknowledges support from the Alexander von Humboldt foundation. E.S.T. acknowledges the partial support by the Simons Foundation Award No.~663281 granted to the Institute of Mathematics of the Polish Academy of Sciences for the years 2021-2023. The authors would like to thank the Isaac Newton Institute for Mathematical Sciences, Cambridge, for support and hospitality during the programme ``Mathematical aspects of turbulence: where do we stand?'' where part of this work was undertaken and supported by EPSRC grant no.~EP/K032208/1.

\section*{Declaration} 
The authors declare that they have no competing interests. Data sharing is not applicable to this article as no datasets were generated or analysed during the current study.
\begin{appendices}

\section{Essential estimates regarding the Besov norm} \label{besovestimatesappendix}

The following estimates are used several times in this paper.
\begin{lemma} \label{lemma:essential-besov}
    For any $1\leq p,q,q_1,q_2\leq \infty$, $\alpha\in \mathbb{R}$, $\theta>0$ the following estimates hold
    \begingroup
    \allowdisplaybreaks
    \begin{align*}
        \lVert f \rVert_{B^0_{p,\infty}} &\lesssim\lVert f \rVert_{L^p} \lesssim \lVert f \rVert_{B^0_{p,1}} , \\
        \lVert f \rVert_{B^\alpha_{p,q_1}} &\lesssim \lVert f \rVert_{B^{\alpha+\theta}_{p,q_2}}, \\
        \lVert f \rVert_{B^\alpha_{p,q_1}} &\lesssim \lVert f \rVert_{B^{\alpha}_{p,q_2}}, \quad \text{ if }q_1\geq q_2 \\
        \lVert \partial_i f \rVert_{B^{\alpha-1}_{p,q}} &\lesssim \lVert f \rVert_{B^\alpha_{p,q}}.
    \end{align*}
    \endgroup
\end{lemma}

For the proof we refer to \cite{sawano} Propositions 2.1, 2.2, 2.3 and Theorem 2.2.

\section{Review of paradifferential calculus} \label{parareview}

We briefly recall some basic notions of paradifferential calculus, further details can be found in \cite{bahouri,mourrat1,mourrat2}. We first introduce some basic notions of Littlewood-Paley theory, which can again by found in \cite{bahouri}. Let $\rho$ be a smooth function with support in the annulus of radii $\frac{3}{4}$ and $\frac{8}{3}$. We let $\{\rho_j\}_{j=-1}^\infty$ be a dyadic partition of unity, i.e.
\begin{equation*}
\rho_0 (\xi) = \rho (\xi), \quad \rho_j (\xi) = \rho (2^{-j} \xi) \text{ for }j=1,2,\ldots, \quad \rho_{-1} (\xi) = 1 - \sum_{j=0}^\infty \rho_j (\xi).
\end{equation*}
Then the Littlewood-Paley blocks are defined as follows for $f \in \mathcal{S}' (\mathbb{T}^3)$
\begin{equation*}
\widehat{\Delta_j f} (\xi) = \rho_j (\xi) \widehat{f} (\xi), \quad j=-1,0, \ldots.
\end{equation*}
Bony's decomposition of a product is formally given by
\begin{equation*}
f g = T_f g + T_g f + R(f,g).
\end{equation*}
Here $T_f g$ and $T_g f$ are called the paraproducts which are given by
\begin{align*}
T_f g = \sum_{j=-1}^\infty \sum_{i=-1}^{j-2} \Delta_i f \Delta_j g, \quad T_g f = \sum_{j=-1}^\infty \sum_{i=-1}^{j-2} \Delta_i g \Delta_j f.
\end{align*}
The resonance term $R(f,g)$ is given by
\begin{equation*}
R(f,g) = \sum_{\lvert k - j \rvert \leq 1} \Delta_k f \Delta_j g.
\end{equation*}
For the paraproducts we have the following estimates.
\begin{lemma}[Lemma 2.1 in \cite{promel}] \label{paraproduct}
Let $\alpha,\beta\in \mathbb{R}$ and $1\leq p,p_1,p_2,q,q_1,q_2\leq \infty$ with
$$
    \frac{1}{p} = \frac{1}{p_1} + \frac{1}{p_2}, \quad \frac{1}{q} = \frac{1}{q_1} + \frac{1}{q_2}.
$$
\begin{itemize}
    \item For any $f \in L^{p_1} (\mathbb{T}^3) $ and $g \in B^{\beta}_{p_2,q} (\mathbb{T}^3)$ we have that
    \begin{equation*}
        \lVert T_f (g) \rVert_{B^\beta_{p,q}} \lesssim \lVert f \rVert_{L^{p_1}} \lVert g \rVert_{B^\beta_{p_2,q}}.
    \end{equation*}

    \item If $\alpha  < 0$ then for any $f \in B^{\alpha}_{p_1, q_1} (\mathbb{T}^3)$ and $g \in B^\beta_{p_2,q_2} (\mathbb{T}^3)$ we have
    \begin{equation*}
        \lVert T_f (g) \rVert_{B^{\alpha + \beta}_{p,q}} \lesssim \lVert f \rVert_{B^\alpha_{p_1,q_1}} \lVert g \rVert_{B^\beta_{p_2,q_2}}.
    \end{equation*}
\end{itemize}
\end{lemma}

For the resonance term the following estimate holds.
\begin{lemma}[Theorem 2.85 in \cite{bahouri}] \label{resonance}
Let $\alpha,\beta\in \mathbb{R}$ and $1\leq p,p_1,p_2,q,q_1,q_2\leq \infty$ with
$$
    \frac{1}{p} = \frac{1}{p_1} + \frac{1}{p_2}, \quad \frac{1}{q} = \frac{1}{q_1} + \frac{1}{q_2}.
$$
\begin{itemize}
    \item If $\alpha+\beta>0$ then we have for any $f \in B^{\alpha}_{p_1,q_1} (\mathbb{T}^3)$ and $g \in B^{\beta}_{p_2,q_2} (\mathbb{T}^3)$
    $$
        \lVert R(f,g) \rVert_{B^{\alpha+\beta}_{p,q}} \lesssim \lVert f \rVert_{B^{\alpha}_{p_1,q_1}} \lVert g \rVert_{B^{\beta}_{p_2,q_2}}.
    $$

    \item If $\alpha+\beta=0$ and $q=1$ then we have for any $f \in B^{\alpha}_{p_1,q_1} (\mathbb{T}^3)$ and $g \in B^{\beta}_{p_2,q_2} (\mathbb{T}^3)$
    $$
        \lVert R(f,g) \rVert_{B^{0}_{p,\infty}} \lesssim \lVert f \rVert_{B^{\alpha}_{p_1,q_1}} \lVert g \rVert_{B^{\beta}_{p_2,q_2}}.
    $$
\end{itemize}
\end{lemma}
For further details and proofs we refer to \cite{bahouri,promel}.

Combining Lemmas~\ref{paraproduct} and \ref{resonance} leads to the following result.
\begin{lemma} \label{lemma:paradiff-summary}
    Let $\alpha < \beta$, $\beta + \alpha > 0$, $\theta>0$, and $1 \leq p_1, p_2, p, q_1,q_2,q \leq \infty$ with
    \begin{equation*}
        \frac{1}{p_1} + \frac{1}{p_2} = \frac{1}{p}.
    \end{equation*}
    \begin{itemize}
    \item For any $f \in B^\alpha_{p_1,q_1} (\mathbb{T}^3)$ and $g \in B^\beta_{p_2,q_2} (\mathbb{T}^3)$ we have that
    \begin{equation*}
        \lVert fg \rVert_{B^\alpha_{p,q_1}} \lesssim \lVert f \rVert_{B^\alpha_{p_1,q_1}} \lVert g \rVert_{B^\beta_{p_2,q_2}}.
    \end{equation*}

    \item If $\alpha>0$ we have for any $f \in B^\alpha_{p_1,q} (\mathbb{T}^3)$ and $g \in B^\alpha_{p_2,q} (\mathbb{T}^3)$
    \begin{equation*}
        \lVert fg \rVert_{B^\alpha_{p,q}} \lesssim \lVert f \rVert_{B^\alpha_{p_1,q}} \lVert g \rVert_{B^\alpha_{p_2,q}}.
    \end{equation*}

    \item If $\alpha>0$ we have for any $f \in B^\theta_{p_1,q_1} (\mathbb{T}^3) \cap B^{\alpha+\theta}_{p_2,q_2} (\mathbb{T}^3)$
    \begin{equation*}
        \lVert f^2 \rVert_{B^\alpha_{p,q}} \lesssim \lVert f \rVert_{B^\theta_{p_1,q_1}} \lVert f \rVert_{B^{\alpha+\theta}_{p_2,q_2}}.
    \end{equation*}
    \end{itemize}
\end{lemma}

\begin{proof}
The estimates claimed in Lemma~\ref{lemma:paradiff-summary} are a simple consequence of Lemmas~\ref{paraproduct} and \ref{resonance}. We only show the third bullet point in more detail as this is the most important estimate needed for this paper. The other estimates can be proven similarly.

By Bony's decomposition we may write
$$
\lVert f^2 \rVert_{B^\alpha_{p,q}} \leq 2\lVert T_f f \rVert_{B^\alpha_{p,q}} + \lVert R(f,f) \rVert_{B^\alpha_{p,q}}.
$$
Lemma~\ref{paraproduct} yields together with Lemma~\ref{lemma:essential-besov}
$$
    \lVert T_f f \rVert_{B^\alpha_{p,q}} \lesssim \lVert f \rVert_{L^{p_1}} \lVert f \rVert_{B^{\alpha}_{p_2,q}} \lesssim \lVert f \rVert_{B^0_{p_1,1}} \lVert f \rVert_{B^{\alpha}_{p_2,q}} \lesssim \lVert f \rVert_{B^\theta_{p_1,q_1}} \lVert f \rVert_{B^{\alpha+\theta}_{p_2,q_2}}.
$$
Moreover we find by using Lemmas~\ref{resonance} and \ref{lemma:essential-besov}
$$
    \lVert R(f,f) \rVert_{B^\alpha_{p,q}} \lesssim \lVert f \rVert_{B^{0}_{p_1,\infty}} \lVert f \rVert_{B^{\alpha}_{p_2,q}}\lesssim \lVert f \rVert_{B^\theta_{p_1,q_1}} \lVert f \rVert_{B^{\alpha+\theta}_{p_2,q_2}}.
$$
\end{proof}

\section{Proof of Proposition \ref{directionlemma} by using commutator estimates} \label{commutatorappendix}
The results in this paper were proven by using an equation of local energy balance and subsequently showing that the defect term in this equation is zero. Most of the results could also have been proven by using commutator estimates (of the type first introduced in \cite{constantin}).

In this section, we will present a proof of Proposition \ref{directionlemma} to illustrate how one can prove these results by using commutator estimates. We repeat the statement of Proposition \ref{directionlemma} here.
\begin{proposition}[Horizontal and vertical Onsager exponents]
Let $\mathbf{u}$ be a type I weak solution of the hydrostatic Euler equations such that $u,v \in L^3 ((0,T); \linebreak B^\alpha_{3,\infty} (\mathbb{T} ; B^\beta_{3, \infty} (\mathbb{T}^2)))$ with $\beta > \frac{2}{3}, \alpha > \frac{1}{3}$ and $2 \alpha + \beta > 2$. We recall that the regularity assumption means that $u$ and $v$ have $B^\alpha_{3,\infty} (\mathbb{T})$ regularity in the $z$-direction, while they have $B^\beta_{3,\infty} (\mathbb{T}^2)$ regularity in the horizontal directions. Under these assumptions the weak solution conserves energy.
\end{proposition}
\begin{proof}
We restrict to the case $\beta > 1$, as otherwise by the assumptions of the theorem we would need $\alpha > \frac{1}{2}$. Then both $\alpha$ and $\beta$ are bigger than $\frac{1}{2}$ and the conclusion of the proposition was already proven in Proposition \ref{exponentlemma}.

We take $\phi_1 = (u^\epsilon)^\epsilon$ and $\phi_2 = (v^\epsilon)^\epsilon$ in the weak formulation in Definition \ref{weaksolutiondefinition}, transfer one of the mollifications and add the equations together. This gives
\begin{align*}
&\int_0^T \int_{\mathbb{T}^3} u^\epsilon \partial_t u^\epsilon d x \; dt + \int_0^T \int_{\mathbb{T}^3} v^\epsilon \partial_t v^\epsilon d x \; dt + \int_0^T \int_{\mathbb{T}^3} (u \textbf{u})^\epsilon \cdot \nabla u^\epsilon d x \; dt + \int_0^T \int_{\mathbb{T}^3} (v \textbf{u} )^\epsilon \cdot \nabla v^\epsilon d x \; dt  \\
&- \int_0^T \int_{\mathbb{T}^3} \Omega u^\epsilon v^\epsilon d x \; dt + \int_0^T \int_{\mathbb{T}^3} \Omega v^\epsilon u^\epsilon d x \; dt + \int_0^T \int_{\mathbb{T}^3} p^\epsilon \partial_x u^\epsilon d x \; dt + \int_0^T \int_{\mathbb{T}^3} p^\epsilon \partial_y v^\epsilon d x \; dt = 0.
\end{align*}
We observe that
\begin{equation*}
-\int_0^T \int_{\mathbb{T}^3} \Omega u^\epsilon v^\epsilon d x \; dt + \int_0^T \int_{\mathbb{T}^3} \Omega v^\epsilon u^\epsilon d x \; dt = 0.
\end{equation*}
Regarding the pressure terms, we obtain
\begin{align*}
&\int_0^T \int_{\mathbb{T}^3} p^\epsilon \partial_x u^\epsilon d x \; dt + \int_0^T \int_{\mathbb{T}^3} p^\epsilon \partial_y v^\epsilon d x \; dt \\
&= \int_0^T \int_{\mathbb{T}^3} p^\epsilon \partial_x u^\epsilon d x \; dt + \int_0^T \int_{\mathbb{T}^3} p^\epsilon \partial_y v^\epsilon d x \; dt + \int_0^T \int_{\mathbb{T}^3} p^\epsilon \partial_z w^\epsilon d x \; dt = 0.
\end{align*}
Concerning the advective terms, we only present the details in handling the integral $\int_0^T \int_{\mathbb{T}^3} (u \textbf{u})^\epsilon \cdot \nabla u^\epsilon d x \; dt$, as the derivation for the other advective term proceeds in a similar fashion. First we look at the first and second terms of the divergence, we get that
\begingroup
\allowdisplaybreaks
\begin{align*}
&\int_0^T \int_{\mathbb{T}^3} (u^2)^\epsilon \partial_x u^\epsilon d x \; dt + \int_0^T \int_{\mathbb{T}^3} (u v)^\epsilon \partial_y u^\epsilon d x \; dt \\
&= \int_0^T \int_{\mathbb{T}^3} \bigg[ u^\epsilon u^\epsilon + \int_{\mathbb{R}^3} \big( \varphi_\epsilon (y) (\delta u (-y; x,t) \otimes \delta u (-y; x,t) ) \big) dy - (u - u^\epsilon) \otimes (u - u^\epsilon) \bigg] \partial_x u^\epsilon d x \; dt \\
&+ \int_0^T \int_{\mathbb{T}^3} \bigg[ u^\epsilon v^\epsilon + \int_{\mathbb{R}^3} \big(\varphi_y (y) \delta u (-y; x,t) \otimes \delta v (-y; x,t) \big) dy  - (u - u^\epsilon) \otimes (v - v^\epsilon) \bigg] \partial_y u^\epsilon d x \; dt.
\end{align*}
\endgroup
In the above we have used commutator estimates as introduced in \cite{constantin,titi2018}. Then we do the estimates
\begingroup
\allowdisplaybreaks
\begin{align*}
&\int_0^T \int_{\mathbb{T}^3} \bigg[  \int_{\mathbb{R}^3} \big( \varphi_\epsilon (y) (\delta u (-y; x,t) \otimes \delta u (-y; x,t) ) \big) dy - (u - u^\epsilon) \otimes (u - u^\epsilon) \bigg] \partial_x u^\epsilon d x \; dt \\
&+ \int_0^T \int_{\mathbb{T}^3} \bigg[  \int_{\mathbb{R}^3} \big(\varphi_y (y) \delta u (-y; x,t) \otimes \delta v (-y; x,t) \big) dy  - (u - u^\epsilon) \otimes (v - v^\epsilon) \bigg] \partial_y u^\epsilon d x \; dt \\
&\leq \bigg[ \epsilon^{2 \alpha} \lVert u \rVert_{B^\alpha_z (B^\beta_h)}^2 + \epsilon^{2 \alpha} \lVert u \rVert_{B^\alpha_z (B^\beta_h)} \lVert v \rVert_{B^\alpha_z (B^\beta_h)} \bigg] \epsilon^{\alpha - 1} \lVert u \rVert_{B^\alpha_z (B^\beta_h)} \\
&\leq \epsilon^{3 \alpha - 1} \lVert u \rVert_{B^\alpha_z (B^\beta_h)}^2 \bigg[ \lVert u \rVert_{B^\alpha_z (B^\beta_h)} + \lVert v \rVert_{B^\alpha_z (B^\beta_h)} \bigg].
\end{align*}
\endgroup
Recall that we assumed that $\alpha > \frac{1}{3}$, therefore these terms go to zero as $\epsilon \rightarrow 0$.

Subsequently, we look at the third component. By using a commutator estimate, we see that
\begin{align*}
\int_0^T \int_{\mathbb{T}^3} (u w)^\epsilon \partial_z u^\epsilon d x \; dt &= \int_0^T \int_{\mathbb{T}^3} \bigg[ u^\epsilon w^\epsilon + \int_{\mathbb{R}^3} \big( \varphi_\epsilon (y) (\delta u (-y; x,t) \otimes \delta w (-y; x,t) ) \big) dy \\
&- (u - u^\epsilon) \otimes (w - w^\epsilon) \bigg] \partial_z u^\epsilon d x \; dt.
\end{align*}
We now make the estimate
\begin{align*}
&\int_0^T \int_{\mathbb{T}^3} \bigg[  \int_{\mathbb{R}^3} \big( \varphi_\epsilon (y) (\delta u (-y; x,t) \otimes \delta w (-y; x,t) ) \big) dy - (u - u^\epsilon) \otimes (w - w^\epsilon) \bigg] \partial_z u^\epsilon d x \; dt \\
&\leq 2\epsilon^\alpha \lVert u \rVert_{B^\alpha_z (B^\beta_h)} \cdot \epsilon^{\beta - 1} \lVert w \rVert_{B^{\alpha+1}_z (B^{\beta-1}_h)} \cdot \epsilon^{\alpha - 1} \lVert u \rVert_{B^\alpha_z (B^\beta_h)} \leq \epsilon^{2 \alpha + \beta - 2} \lVert u \rVert_{B^\alpha_z (B^\beta_h)}^2 \lVert w \rVert_{B^{\alpha+1}_z (B^{\beta-1}_h)}.
\end{align*}
Recall that we assumed that $2 \alpha + \beta > 2$, so therefore the terms go to zero as $\epsilon \rightarrow 0$. Thus we are left with the terms
\begin{align*}
&\int_0^T \int_{\mathbb{T}^3} \bigg( u^\epsilon \bigg[ u^\epsilon \partial_x u^\epsilon + v^\epsilon \partial_y u^\epsilon + w^\epsilon \partial_z u^\epsilon \bigg] \bigg) dx dt = \frac{1}{2} \int_0^T \int_{\mathbb{T}^3} \textbf{u}^\epsilon \cdot \nabla \big( (u^\epsilon)^2 \big) dx dt = 0.
\end{align*}
As a result, we get that
\begin{align*}
\big\lvert \lVert u^\epsilon (\cdot, T) \rVert_{L^2} + \lVert v^\epsilon (\cdot, T) \rVert_{L^2} - \lVert u^\epsilon (\cdot, 0) \rVert_{L^2} - \lVert u^\epsilon (\cdot, 0) \rVert_{L^2} \big\rvert \xrightarrow[]{\epsilon \rightarrow 0} 0.
\end{align*}
Hence energy is conserved.
\end{proof}
\end{appendices}
\bibliography{hydrostatic_Euler_final}

\end{document}